\documentclass{article}
\usepackage{amsthm,amsmath,amsfonts}
\usepackage{verbatim,xcolor} 
\usepackage{graphicx}
\usepackage[lmargin=.7in,rmargin=.7in,tmargin=1in,bmargin=1in]{geometry} 
\usepackage{authblk}
\usepackage[normalem]{ulem}
\usepackage{stmaryrd}
\usepackage{marginnote}
\usepackage{mathrsfs}
\usepackage{amscd}
\usepackage{verbatim}
\usepackage{subcaption}

\DeclareMathOperator{\BigO}{{O}}
\newcommand{\RR}{{\mathbb{R}}}
\newcommand{\BB}{{\mathbb{B}}}
\newcommand{\tu}{{\tilde{u}}}

\newcommand{\Mrad}{{M_{\textup{rad}}}}

\newcommand{\Vrad}{{V_{\textup{rad}}}}

\newtheorem{thm}{Theorem}
\newtheorem{lem}{Lemma}
\newtheorem{proposition}{Proposition}
\newtheorem{corollary}{Corollary}

\newtheorem{remark}{Remark}

\newtheorem{innercustomthm}{Theorem}
\newenvironment{customthm}[1]
  {\renewcommand\theinnercustomthm{#1}\innercustomthm}
  {\endinnercustomthm}

\theoremstyle{Remark}

\makeatletter
\newcommand{\subjclass}[2][1991]{%
  \let\@oldtitle\@title%
  \gdef\@title{\@oldtitle\footnotetext{#1 \emph{Mathematics subject classification.} #2}}%
}
\newcommand{\keywords}[1]{%
  \let\@@oldtitle\@title%
  \gdef\@title{\@@oldtitle\footnotetext{\emph{Key words and phrases.} #1.}}%
}
\makeatother

\begin{document}
\title{A Sharp Isoperimetric Inequality for the Second Eigenvalue of the Robin Plate}
\author[$\dagger$]{L. Mercredi Chasman}
\author[$\star$]{Jeffrey J. Langford}

\affil[$\dagger$]{University of Minnesota - Morris, 600 E. 4th Street, Morris, MN 56267, USA}
\affil[$\star$]{Bucknell University, 1 Dent Drive, Lewisburg, PA 17837, USA}

\date{\today}

\keywords{Bilaplacian, Robin boundary conditions, isoperimetric inequality, Bessel functions}
 
\subjclass[2020]{Primary 35P15; Secondary 35J40, 74K20, 33C10.}

\maketitle

\abstract{Among all $C^{\infty}$ bounded domains with equal volume, we show that the second eigenvalue of the Robin plate is uniquely maximized by an open ball, so long as the Robin parameter lies within a particular range of negative values. Our methodology combines recent techniques introduced by Freitas and Laugesen to study the second eigenvalue of the Robin membrane problem and techniques employed by Chasman to study the free plate problem. In particular, we choose eigenfunctions of the ball as trial functions in the Rayleigh quotient for a general domain; such eigenfunctions are comprised of ultraspherical Bessel and modified Bessel functions. Much of our work hinges on developing an understanding of delicate properties of these special functions, which may be of independent interest.}

\section{Introduction}
Our work is motivated by the Robin membrane problem, which recently has attracted considerable attention. To set the stage, let $\Omega\subset\RR^d$ be a $C^{\infty}$ bounded domain with $d \geq 2$. The Robin eigenvalue problem for the Laplacian is then
\begin{equation}\label{eqn:robinmembrane}
\begin{cases}
-\Delta u =\lambda u  &\text{in $\Omega$},\\
\alpha u+\frac{\partial u}{\partial n}=0 &\text{on $\partial\Omega$},
\end{cases}
\end{equation}
where $\alpha$ is a real parameter. The eigenvalues of problem \eqref{eqn:robinmembrane} are known to satisfy
\[
\lambda_1(\Omega;\alpha)< \lambda_2(\Omega;\alpha)\leq \lambda_3(\Omega;\alpha) \leq \cdots \to +\infty.
\]
The parameter $\alpha$ allows for interpolation among several membrane problems with different boundary conditions. By considering $\alpha=0$, one recovers the free membrane (Neumann) problem; sending $\alpha\to\infty$ recovers the fixed membrane (Dirichlet) problem; taking $\lambda=0$ and treating $-\alpha$ as an eigenvalue recovers the Steklov problem. Articles on eigenvalue estimates and related questions for the Robin problem \eqref{eqn:robinmembrane} fill many journal pages; we refer the interested reader to the wonderful survey (and references therein) by Bucur, Freitas, and Kennedy \cite[Ch.4]{Henrot}. As motivation for our work, we recall a specific result due to Freitas and Laugesen \cite{FreitasLaugesen1}:

\begin{customthm}{A}\label{Thm:A}
With $\Omega\subset\RR^d$ ($d\geq 2$) a $C^{\infty}$ bounded domain, let $\Omega^*\subset\RR^d$ denote an open ball with the same volume as $\Omega$. Then
\[
\lambda_2(\Omega;\alpha)\leq \lambda_2(\Omega^*;\alpha),\quad \alpha \in \left[-\frac{d+1}{d}R^{-1},0\right],
\]
where $R$ is the radius of $\Omega^*$. Equality holds if and only if $\Omega$ is a ball.
\end{customthm}

The main result of our paper is a plate-analogue of Theorem \ref{Thm:A}. The Robin plate eigenvalue problem with tension/compression $\tau$ and Robin parameter $\alpha \in\RR$ is
\begin{equation}\label{eqn:robinplate}
\begin{cases}
\Delta^2 u-\tau\Delta u=\Lambda u &\text{in $\Omega$},\\
Mu:=\frac{\partial^2u}{\partial n^2}=0  &\text{on $\partial\Omega$},\\
Vu:=\tau\frac{\partial u}{\partial n}-\frac{\partial \Delta u}{\partial n}-\mathrm{div}_{\partial\Omega}[P_{\partial\Omega}((D^2u)n)]+\alpha u=0 &\text{on $\partial\Omega$.}
\end{cases}
\end{equation}
Here $P_{\partial\Omega}$ denotes the projection onto the tangent space of ${\partial\Omega}$, and $\mathrm{div}_{\partial\Omega}$ is the surface divergence. We note that when $\alpha=0$, we recover the free plate problem; when $\Lambda=0$ and we view $-\alpha$ as an eigenvalue, we recover the biharmonic Steklov problem (see also Remark \ref{Rmk:GenRobRmk} for a more general Robin plate problem that has a connection to the clamped plate problem).

The spectrum for problem \eqref{eqn:robinplate} consists of real eigenvalues satisfying (see Proposition \ref{prop:SpecProps})
\begin{equation}\label{Ineqs:RobSpec}
\Lambda_1(\Omega;\tau,\alpha)\leq\Lambda_2(\Omega;\tau,\alpha)\leq \Lambda_3(\Omega;\tau,\alpha)\leq\cdots \to +\infty.
\end{equation}
The main result of this paper is an isoperimetric inequality for $\Lambda_2$:
\begin{thm}\label{thm:mainthm} Let $\Omega$ and $\Omega^*$ be as in Theorem \ref{Thm:A}. Then for all $\tau>0$, we have
\[
\Lambda_2(\Omega;\tau,\alpha)\leq \Lambda_2(\Omega^*;\tau,\alpha),\qquad \alpha \in \left[-\frac{\tau}{R},0\right],
\]
where $R$ is the radius of $\Omega^*$. Moreover, when $-\tau/R<\alpha<0$, equality holds precisely when $\Omega$ is a ball.
\end{thm}

In \cite{FreitasLaugesen1}, Freitas and Laugesen show that Theorem \ref{Thm:A} interpolates between isoperimetric inequalities for the free membrane and Steklov eigenvalue problems. Likewise, Theorem \ref{thm:mainthm} gives isoperimetric inequalities for  the biharmonic free plate and Steklov eigenvalue problems; to state our corollaries, we require some additional notation. The free plate eigenvalue problem is obtained by taking $\alpha=0$ in problem \eqref{eqn:robinplate}:
\begin{equation}\label{eqn:FreePlate}
\begin{cases}
\Delta^2 u-\tau\Delta u=\omega u &\text{in $\Omega$},\\
\frac{\partial^2u}{\partial n^2}=0  &\text{on $\partial\Omega$},\\
\tau\frac{\partial u}{\partial n}-\frac{\partial \Delta u}{\partial n}-\mathrm{div}_{\partial\Omega}[P_{\partial\Omega}((D^2u)n)]=0 &\text{on $\partial\Omega$.}
\end{cases}
\end{equation}
The spectrum of this problem (see \cite{chasmanineq}) satisfies
\[
0= \omega_1(\Omega)\leq \omega_2(\Omega) \leq \omega_3(\Omega)\leq \cdots \to +\infty,
\]
and it is well known that the second eigenvalue satisfies $\omega_2(\Omega)>0$ so long as $\tau>0$. Similarly, the biharmonic Steklov problem is obtained from \eqref{eqn:robinplate} by setting $\Lambda=0$ and $\alpha=-\sigma$:
\begin{equation}\label{eqn:BiharmonicSteklov}
\begin{cases}
\Delta^2 u-\tau\Delta u=0 &\text{in $\Omega$},\\
\frac{\partial^2u}{\partial n^2}=0  &\text{on $\partial\Omega$},\\
\tau\frac{\partial u}{\partial n}-\frac{\partial \Delta u}{\partial n}-\mathrm{div}_{\partial\Omega}[P_{\partial\Omega}((D^2u)n)]=\sigma u &\text{on $\partial\Omega$.}
\end{cases}
\end{equation}
The spectrum of this problem is also well understood (see \cite{BP}):
\[
0=\sigma_1(\Omega)<\sigma_2(\Omega)\leq \sigma_3(\Omega) \leq \cdots \to +\infty.
\]
The main corollary of our paper recovers classical isoperimetric inequalities for the lowest nonzero eigenvalues of the free plate and Steklov eigenvalue problems (see \cite{BP, chasmanineq}):
\begin{corollary}\label{Cor:MainCor}
Let $\Omega$ and $\Omega^*$ be as in Theorem \ref{thm:mainthm}. If $\tau>0$, then the lowest nonzero eigenvalue of the free plate problem \eqref{eqn:FreePlate} satisfies
\[
\omega_2(\Omega)\leq\omega_2(\Omega^*),
\]
and the first nonzero Steklov eigenvalue of problem \eqref{eqn:BiharmonicSteklov} satisfies
\[
\sigma_2(\Omega)\leq\sigma_2(\Omega^*).
\]
\end{corollary}

To place our work in the existing literature, we focus our attention on isoperimetric inequalities for plate problems and related membrane results. Generally speaking, isoperimetric inequalities for plates are harder to establish than their membrane counterparts (and appear much later in the literature). For example, in the late 1800's Lord Rayleigh \cite{RToS} conjectured that the fundamental frequencies of a fixed membrane and a clamped plate are minimized by disks (and more generally by balls in higher dimensions). The conjecture for membranes was proven (independently) in the 1920's by Faber and Krahn \cite{Faber, Krahn}. By contrast, significant progress towards Rayleigh's conjecture for the clamped plate didn't begin until the 1980's with the work of Talenti \cite{T81}. Building on Talenti's work, Nadirashvili \cite{N95} later proved Rayleigh's conjecture for the clamped plate in dimension $d=2$, and Ashbaugh and Benguria \cite{AB95}, also building on Talenti's work, established the conjecture independently in dimensions $d=2$ and $d=3$. The conjecture still remains open for higher dimensions, although work of Ashbaugh and Laugesen \cite{AL96} indicates that the conjecture is ``asymptotically true in high dimensions.''

Likewise, in the 1950's Szeg\H{o} \cite{Szego} and Weinberger \cite{Weinberger} proved that the second frequency (i.e. the first nonzero frequency) of a free membrane is maximized by a ball, and it wasn't until fairly recently that the analogous plate result was established by Chasman \cite{chasmanineq}; see also \cite{BCP, chasmanball, ChasmanPR}. For the Steklov Laplacian, the first isoperimetric inequality dates to Weinstock \cite{Weinstock} in dimension $d=2$ (using a perimeter constraint), with later work in all dimensions (under volume constraint) by Brock \cite{Brock}. The analogous isoperimetric inequality for the biharmonic operator was established by Buoso and Provenzano \cite{BP};  see also \cite{BFG, BG, BCP}.

The first isoperimetric inequality for the Robin Laplacian came in the 1980's with the work of Bossel \cite{Bossel}. For $\alpha>0$, Bossel showed that in dimension $d=2$, the first Robin eigenvalue is minimized by a disk, and Daners \cite{Daners} later extended this result to all dimensions. See also the related work \cite{BucurGiacomini1}. For $\alpha<0$, Bareket \cite{Bareket} conjectured  that balls maximize the first Robin eigenvalue. For such $\alpha$, Ferone, Nitsch, and Trombetti \cite{FNT} showed that balls are local maximizers of the first eigenvalue. In  \cite{FK}, Freitas and Krej\v{c}i\v{r}\'{\i}k disproved Bareket's conjecture in general, but showed that disks maximize the first eigenvalue when the dimension $d=2$ and $\alpha$ is sufficiently close to zero. In \cite{AFK}, Antunes, Freitas, and Krej\v{c}i\v{r}\'{\i}k proved Baraket's conjecture when $d=2$ under perimeter (rather than volume) constraint. Also under perimeter constraint,  Bucur, Ferone, Nitsch, and Trombetti \cite{BFNT}  proved Bareket's conjecture in all dimensions among convex domains. For the second Robin eigenvalue, Laugesen and Freitas \cite{FreitasLaugesen1} proved that balls maximize the second eigenvalue for a range of negative $\alpha$ values (see Theorem \ref{Thm:A} above); see also \cite{FrietasLaugesen2}. For the Robin plate, however, no isoperimetric inequalities appear in the literature. Indeed, we are only aware of one other paper \cite{BK} that discusses Robin plates. In this extensive work, Buoso and Kennedy study a more general Robin problem than \eqref{eqn:robinplate} (i.e. one with more parameters), with an eye towards understanding the asymptotic behavior of eigenvalues as the parameters become large. The authors also establish variational formulae for the eigenvalues and show that balls are critical domains. Such criticality results, however, do not imply global upper bounds like those discussed in Theorem \ref{thm:mainthm}. Thus, our paper represents a new avenue of research within the study of plates.

The remainder of the paper is structured as follows. In Section 2, we establish basic properties of the spectrum of the Robin plate problem. In Section 3, we collect a number of facts about Bessel functions integral to our discussion of eigenvalues and eigenfunctions of the unit ball in Sections 4 and 5. In Section 6, we construct our trial functions and establish a center of mass result. In Section 7, we establish several technical inequalities needed to prove monotonicity of the Rayleigh quotient in Section 8. Also in Section 8, we deduce  isoperimetric inequalities for the free plate and biharmonic Steklov problems from Theorem \ref{thm:mainthm}.

\section{Coercivity, Natural Boundary Conditions, and the Spectrum of the Robin Plate Problem}
Our approach to the Robin plate problem is modeled after Chasman's approach to the free plate problem \cite{chasmanineq}. Thus, we begin with an appropriate Rayleigh quotient and bilinear form. We derive the boundary conditions by closely examining the weak eigenvalue equation and integrating by parts.

The Robin plate Rayleigh quotient is
\begin{equation}\label{eq:RQ}
Q[\Omega; \alpha,u]=Q[u]:=\frac{\int_\Omega \left( |D^2u|^2+\tau |Du|^2\right)\,dx+\alpha\int_{\partial\Omega}u^2\,dS}{\int_\Omega u^2\,dx}, \qquad u\in H^2(\Omega).
\end{equation}

For the time being, we let $\tau$ and $\alpha$ be general real parameters; later, we restrict the range of values that $\tau$ and $\alpha$ assume. The bilinear form associated to the Rayleigh quotient is
\[
a(u,\phi)=\int_\Omega \left(\sum_{i,j}u_{x_ix_j}\phi_{x_ix_j}+\tau Du\cdot D\phi\right)\,dx +\alpha \int_{\partial\Omega} u\phi \,dS,  \qquad u,\phi \in H^2(\Omega).
\]
We denote the bilinear form for the free plate as
\[
a_{\text{free},\tau}(u,\phi)=\int_\Omega \left(\sum_{i,j}u_{x_ix_j}\phi_{x_ix_j}+\tau Du\cdot D\phi\right)\,dx, \qquad u,\phi \in H^2(\Omega).
\]
This notation allows us to express the Robin plate form as the sum of the free plate form and a surface integral term.

Our immediate task is to show that the eigenvalues of the Robin plate problem \eqref{eqn:robinplate} behave as in \eqref{Ineqs:RobSpec}. We accomplish this task by showing that the associated form is both continuous and coercive.

\begin{proposition}\label{prop:SpecProps}
Let $\Omega \subset \mathbb{R}^d$ be a $C^{\infty}$ bounded domain. There exist positive constants $K_1, K_2$, and $K_3$ such that
\[
K_1\|u\|_{H^2(\Omega)}^2 \leq a(u,u)+K_2\|u\|_{L^2(\Omega)}^2\leq K_3\|u\|_{H^2(\Omega)}^2.
\]
Thus, the eigenvalues of the operator associated to $a(\cdot,\cdot)$ have finite multiplicity and satisfy
\begin{equation}\label{SpIneqs}
\Lambda_1(\Omega;\tau,\alpha)\leq\Lambda_2(\Omega;\tau,\alpha)\leq \Lambda_3(\Omega;\tau,\alpha)\leq\cdots \to +\infty.
\end{equation}
Weak eigenfunctions, moreover, form an orthonormal basis of $L^2(\Omega)$, are smooth on $\overline{\Omega}$, and satisfy the eigenvalue problem \eqref{eqn:robinplate} in the classical sense.
\end{proposition}

\begin{proof}
To show that $a(\cdot,\cdot)$ is coercive for all $\tau,\alpha \in \mathbb{R}$, we rely on the coercivity the free plate form. First note that for $u\in H^2(\Omega)$,
\[
a(u,u)\geq ||D^2u||_{L^2(\Omega)}^2+\tau ||Du\|_{L^2(\Omega)}^2 - |\alpha|  \cdot ||u||_{L^2(\partial \Omega)}^2.
\]
By the Trace Theorem, there exists $C>0$ such that
\begin{equation}\label{eq:tr}
\|Tu\|_{L^2(\partial\Omega)}^2\leq C\|u\|_{H^1(\Omega)}^2.
\end{equation}
Combining the two inequalities immediately above, we see
\begin{align}
a(u,u)+K_2\|u\|_{L^2(\Omega)}^2&\geq ||D^2 u||_{L^2(\Omega)}^2+(\tau-|\alpha|C)||Du||_{L^2(\Omega)}^2+(K_2-|\alpha|C)||u||_{L^2(\Omega)}^2\nonumber \\
&=a_{\textup{free},\tau-|\alpha|C}(u,u)+(K_2-|\alpha|C)||u||_{L^2(\Omega)}^2\label{Ineq:FP1}.
\end{align}
By the coercivity of the free plate form $a_{\textup{free},\tau-|\alpha|C}(\cdot,\cdot)$, there exist positive constants $C_1$ and $C_2$ (see Proposition 1 of \cite{chasmanthesis}) where
\[
C_1\|u\|^2_{H^2(\Omega)}\leq a_{\textup{free},\tau-|\alpha|C}(u,u)+C_2\|u\|^2_{L^2(\Omega)}.
\]
Choosing $K_2$ large enough so that $K_2-|\alpha|C\geq C_2$, inequality \eqref{Ineq:FP1} and the inequality immediately above together imply
\[
a(u,u)+K_2\|u\|_{L^2(\Omega)}^2 \geq C_1\|u\|^2_{H^2(\Omega)}.
\]
The coercivity of the form $a(\cdot,\cdot)$ follows by taking $K_1=C_1$. The continuity of $a(\cdot,\cdot)$ follows from the Trace inequality \eqref{eq:tr}:
\[
a(u,u)+K_2\|u\|_{L^2(\Omega)}^2 \leq K_3\|u\|^2_{H^2(\Omega)},
\]
where  $K_3=\max \{1,\tau+|\alpha|C,K_2+|\alpha|C\}$.

Since $H^2(\Omega)$ compactly embeds in $L^2(\Omega)$, well-known results (e.g. Corollary 7.D of \cite{Showalter}) imply that $a(\cdot,\cdot)$ has a set of complete real-valued weak eigenfunctions and eigenvalues of finite multiplicity as in \eqref{SpIneqs}. Claimed regularity of the eigenfunctions likewise follows from standard results, see p.668 of \cite{Nirenberg} and Propositions 4.3 of \cite{Taylor}, for example.

Having established that eigenfunctions belong to $C^{\infty}(\overline{\Omega})$, we next show that (weak) eigenfunctions associated to $a(\cdot,\cdot)$ solve \eqref{eqn:robinplate} in the classical sense. For an eigenfunction $u$, the weak eigenvalue equation yields
\begin{equation}\label{eq:weakeval}
a(u,\phi)=\Lambda\int_\Omega u\phi\,dx,\quad\text{for all $\phi\in C^\infty(\overline{\Omega})$}.
\end{equation}
This equation can be rewritten using the free plate's form $a_{\text{free},\tau}(\cdot,\cdot)$ as
\[
a_{\text{free},\tau}(u,\phi)+\alpha \int_{\partial \Omega}u\phi \,dS=\Lambda\int_\Omega u\phi\,dx.
\]
From the proof of Proposition 5 in \cite{chasmanineq}, we can write
\begin{align*}
a_{\text{free},\tau}(u,\phi)-\Lambda\int_\Omega u\phi\,dx&=\int_\Omega(\Delta^2u-\tau \Delta u-\Lambda u)\phi\,dx+\int_{\partial\Omega}\frac{\partial^2u}{\partial n^2}\frac{\partial\phi}{\partial n}\,dS\\
&\qquad+\int_{\partial\Omega}\left(\tau\frac{\partial u}{\partial n}-\frac{\partial \Delta u}{\partial n}-\mathrm{div}_{\partial\Omega}[P_{\partial\Omega}((D^2u)n)]\right)\phi\,dS.
\end{align*}
Thus, \eqref{eq:weakeval} can be rewritten as
\begin{align}\label{eq:finalzero}
0&=a(u,\phi)-\Lambda\int_\Omega u\phi\,dx \nonumber \\
&=\int_\Omega(\Delta^2u-\tau \Delta u-\Lambda u)\phi\,dx +\int_{\partial\Omega}\frac{\partial^2u}{\partial n^2} \frac{\partial \phi}{\partial n}\,dS \nonumber \\
&\qquad+\int_{\partial\Omega}\left(\tau\frac{\partial u}{\partial n}-\frac{\partial \Delta u}{\partial n}-\mathrm{div}_{\partial\Omega}[P_{\partial\Omega}((D^2u)n)]+\alpha u\right)\phi\,dS.
\end{align}
Taking $\phi \in C_c^{\infty}(\Omega)$, \eqref{eq:finalzero} gives that $\Delta^2u-\tau \Delta u-\Lambda u=0$ pointwise in $\Omega$. On the other hand, any smooth function $\phi \in C^{\infty}(\partial \Omega)$ admits an extension to $C^{\infty}(\overline{\Omega})$ satisfying $\frac{\partial \phi}{\partial n}=0$ along $\partial \Omega$. Thus, \eqref{eq:finalzero} implies that $\tau\frac{\partial u}{\partial n}-\frac{\partial \Delta u}{\partial n}-\mathrm{div}_{\partial\Omega}[P_{\partial\Omega}((D^2u)n)]+\alpha u=0$ and $\frac{\partial^2u}{\partial n^2}=0$ pointwise on $\partial \Omega$.
\end{proof}

The final tool of this section is a result which says that the eigenvalues of problem \eqref{eqn:robinplate} change continuously as we vary $\alpha$. This result, and its proof, are inspired by Proposition 16 of \cite{GirourdLaugesen}.

\begin{proposition}\label{Prop:Cont}
Let $\Omega \subset \mathbb{R}^d$ denote a bounded $C^{\infty}$ domain. Then for $\tau>0$, the eigenvalues $\Lambda_k(\Omega; \tau, \alpha)$ of problem \eqref{eqn:robinplate}  are continuous with respect to the Robin parameter $\alpha$.
\end{proposition}
\begin{proof}
We start by recalling the min-max characterization for the eigenvalue $\Lambda_k(\Omega;\tau, \alpha)$:
\[
\Lambda_k(\Omega;\tau, \alpha)=\min_{E\in \mathcal E_k} \max_{u \in E} Q[\Omega;\alpha,u],
\]
where $\mathcal E_k$ denotes the collection of all $k$-dimensional subspaces of $H^2(\Omega)$ and $Q[\Omega;\alpha,u]$ denotes the Rayleigh quotient
\[
Q[\Omega;\alpha,u]= \frac{\int_\Omega \left( |D^2u|^2+\tau |Du|^2\right)\,dx +\alpha \int_{\partial\Omega}u^2\,dS}{\int_{\Omega}u^2\,dx}.
\]
Note that if $\phi_{1,\alpha},\ldots,\phi_{k,\alpha}$ denote an orthonormal set of eigenfunctions corresponding to the eigenvalues $\Lambda_1(\Omega;\tau, \alpha),\ldots,\Lambda_k(\Omega;\tau, \alpha)$ and $E_\alpha=\textup{span}\{\phi_{1,\alpha},\ldots,\phi_{k,\alpha}\}$, then by homogeneity of the Rayleigh quotient
\[
\max_{u\in E_{\alpha}} Q[\Omega;\alpha, u]=\max_{\underset{c_1^2+\cdots +c_k^2=1}{c_1,\ldots ,c_k}}\sum_{j=1}^kc_j^2\Lambda_j(\Omega;\tau, \alpha),
\]
from which we see that $\displaystyle \max_{u\in E_{\alpha}} Q[\Omega;\alpha, u]=\Lambda_k(\Omega;\tau, \alpha)$.

Next, fix a value $\alpha$ and choose a sequence $\alpha_n\to \alpha$. Observe that
\begin{equation*}
\max_{u\in E_{\alpha}} Q[\Omega;\alpha_n, u]\leq \Lambda_k(\Omega; \tau, \alpha) + \max_{\underset{c_1^2+\cdots +c_k^2=1}{c_1,\ldots ,c_k}} (\alpha_n-\alpha)\int_{\partial \Omega} \left( \sum_{j=1}^k c_j\phi_{j,\alpha}\right)^2\,dS,
\end{equation*}
and so we deduce from the min-max formula that
\begin{equation}\label{Ineq:lowerbd}
\limsup_{n\to \infty}\Lambda_k(\Omega; \tau, \alpha_n) \leq \limsup_{n \to \infty} \max_{u \in E_{\alpha}} Q[\Omega;\alpha_n,u]\leq \Lambda_k(\Omega;\tau, \alpha).
\end{equation}

The eigenvalues $\Lambda_j(\Omega; \tau, \alpha_n)$ for $j=1,\ldots,k$ are bounded since $\Lambda(\Omega; \tau, \alpha)$ is monotone as a function of $\alpha$ (courtesy of min-max), and so the coercivity of our form (see Proposition \ref{prop:SpecProps}) implies that the functions $\phi_{j,\alpha_n}$ are bounded in $H^2(\Omega)$. By Banach-Alaoglu, we can pass to a subsequence $\alpha_{n_m}$ so that $\phi_{j,\alpha_{n_m}} \to \phi_j$ weakly in $H^2(\Omega)$. By passing to another subsequence, we can assume by Rellich-Kondrachov that $\phi_{j,\alpha_{n_m}} \to \phi_j$ in $L^2(\Omega)$ and that pointwise a.e. convergence holds. By p.134 of \cite{EvansGariepy}, the inequality
\[
\int_{\partial \Omega}f^2\,dS \leq C\int_{\Omega} \left( |Df|f+f^2\right)\,dx,\quad f\in H^1(\Omega),
\]
implies that $\phi_{j,\alpha_{n_m}} \to \phi_j$ in $L^2(\partial \Omega)$. By $L^2(\Omega)$ convergence, the $\phi_j$ are orthonormal:
\[
\int_{\Omega}\phi_i\phi_j\,dx=\lim_{m\to\infty}\int_{\Omega}\phi_{i,\alpha_{n_m}}\phi_{j,\alpha_{n_m}}\,dx=\delta_{i,j}.
\]
We conclude that $\{\phi_1,\ldots,\phi_k\}$ is $k$-dimensional. Set $E=\textup{span}\{\phi_1,\ldots,\phi_k\}$.

Let $u\in \textup{span}\{\phi_1,\ldots,\phi_k\}$ and write $u=c_1\phi_1+\cdots+c_k\phi_k$. Define
\[
u_{\alpha_{n_m}}=c_1\phi_{1,\alpha_{n_m}}+\cdots +c_k\phi_{k,\alpha_{n_m}}.
\]
Our work above gives that $u_{\alpha_{n_m}} \to u$ in both $L^2(\Omega)$ and $L^2(\partial \Omega)$, and weakly in $H^2(\Omega)$. By weak convergence,
\begin{align*}
\int_{\Omega}u_{x_ix_j}^2\,dx&=\liminf_{m\to \infty}\int_{\Omega}u_{x_ix_j}(u_{\alpha_{n_m}})_{x_ix_j}\,dx\\
&\leq \left( \int_{\Omega}u_{x_ix_j}^2\,dx\right)^{1/2}\liminf_{m\to \infty}\left(\int_{\Omega}(u_{\alpha_{n_m}})^2_{x_ix_j}\,dx\right)^{1/2}
\end{align*}
which implies
\begin{equation}\label{Eqn:Dsquared}
\int_{\Omega}|D^2u|^2\,dx \leq \liminf_{m\to \infty} \int_{\Omega}|D^2u_{\alpha_{n_m}}|^2\,dx.
\end{equation}
A similar argument shows that
\begin{equation}\label{Eqn:GradEst}
\int_{\Omega}|Du|^2\,dx\leq \liminf_{m\to \infty}\int_{\Omega}|Du_{\alpha_{n_m}}|^2\,dx.
\end{equation}
By $L^2$-convergence,
\begin{equation}\label{Eqns:L2convs}
\lim_{m\to \infty}\int_{\Omega}u_{\alpha_{n_m}}^2\,dx=\int_{\Omega}u^2\,dx \quad \textup{and}\quad \lim_{m\to \infty}\int_{\partial \Omega}u_{\alpha_{n_m}}^2\,dS=\int_{\partial \Omega}u^2\,dS.
\end{equation}
Combining \eqref{Eqn:Dsquared}, \eqref{Eqn:GradEst}, and \eqref{Eqns:L2convs}, and using that $\tau> 0$, we conclude
\[
Q[\Omega;\alpha,u]\leq \liminf_{m\to \infty} Q[\Omega;\alpha_{n_m},u_{\alpha_{n_m}}] \leq \liminf_{m\to \infty} \max_{u\in E_{\alpha_{n_m}}} Q[\Omega; \alpha_{n_m},u]=\liminf_{m\to \infty} \Lambda_k(\Omega;\alpha_{n_m}).
\]
Using $E$ as a trial set in the min-max formula for $\Lambda_k(\Omega; \tau, \alpha)$, the above string of inequalities implies
\[
\Lambda_k(\Omega; \tau, \alpha)\leq \max_{u\in E} Q[\Omega;\alpha,u] \leq \liminf_{m\to \infty} \Lambda_k(\Omega; \tau, \alpha_{n_m}).
\]
We claim that the above string of inequalities implies that $\Lambda_k(\Omega; \tau, \alpha)\leq \liminf_{n\to \infty} \Lambda_k(\Omega; \tau, \alpha_n)$. If this were not the case, and instead
$\liminf_{n\to \infty} \Lambda_k(\Omega; \tau, \alpha_n)<\Lambda_k(\Omega; \tau, \alpha)$, we could extract a subsequence $\alpha_n'$ where
\begin{equation}\label{ineq:viol}
\lim_{n\to \infty}\Lambda_k(\Omega; \tau, \alpha_n')<\Lambda_k(\Omega; \tau, \alpha).
\end{equation}
Applying our work above (with $\alpha_n'$ in place of $\alpha_n$), we deduce
\[
\Lambda_k(\Omega; \tau, \alpha)\leq \liminf_{m\to \infty}\Lambda_k(\Omega; \tau, \alpha_{n_m}')
\]
for some subsequence $\alpha_{n_m}'$, which violates \eqref{ineq:viol}. We conclude
\begin{equation}\label{Ineq:NeededUB}
\Lambda_k(\Omega; \tau, \alpha)\leq \liminf_{n\to \infty} \Lambda_k(\Omega; \tau, \alpha_n).
\end{equation}
Combining inequalities \eqref{Ineq:lowerbd} and \eqref{Ineq:NeededUB} gives the result.
\end{proof}

\begin{remark}\label{Rmk:GenRobRmk}
We note that one may generalize the Robin plate problem \eqref{eqn:robinplate}. Indeed, for constants $\alpha,\beta$, consider the bilinear form
\[
a(u,\phi)=\int_\Omega \left( \sum_{i,j}u_{x_ix_j}\phi_{x_ix_j}+\tau Du\cdot D\phi \right)\,dx +\int_{\partial\Omega} \left( \beta Du\cdot D\phi+\alpha u\phi \right)\,dS,\qquad u,\phi\in H^2(\Omega),
\]
with associated Rayleigh quotient
\[
Q[\Omega; \alpha,\beta,u]=Q[u]:=\frac{\int_\Omega \left( |D^2u|^2+\tau |Du|^2\right)\,dx+\int_{\partial\Omega}\left(\beta  |Du|^2+\alpha u^2 \right)\,dS}{\int_\Omega u^2\,dx}, \qquad u\in H^2(\Omega).
\]
When $\beta\geq 0$, the form $a(\cdot,\cdot)$ is coercive and continuous as in Proposition \ref{prop:SpecProps}, and weak eigenfunctions are classical solutions to the PDE problem
\begin{equation*}
\begin{cases}
\Delta^2 u-\tau\Delta u=\Lambda u &\text{in $\Omega$},\\
\frac{\partial^2u}{\partial n^2}+\beta \frac{\partial u}{\partial n}=0  &\text{on $\partial\Omega$},\\
\tau\frac{\partial u}{\partial n}-\frac{\partial \Delta u}{\partial n}-\mathrm{div}_{\partial\Omega}[P_{\partial\Omega}((D^2u)n)]-\beta\Delta_{\partial\Omega}u+\alpha u=0 &\text{on $\partial\Omega$.}
\end{cases}
\end{equation*}
Sending $\alpha \to \infty$, then $\beta \to \infty$ in this more general setting recovers clamped boundary conditions $u=\frac{\partial u}{\partial n}=0$ on $\partial \Omega$. It would be interesting to know what version of Theorem \ref{thm:mainthm} exists for this generalized problem, but this is not something we pursue here; we plan to study this problem in a subsequent paper.
\end{remark}

\section{Bessel Functions}
The radial parts of unit ball eigenfunctions of \eqref {eqn:robinplate} are comprised of linear combinations of ultraspherical Bessel and modified Bessel functions. We collect the definitions and some key properties of these functions here. Ultraspherical Bessel functions can be expressed in terms of ordinary Bessel functions, and so many of their properties (such as recurrence relations) can be derived from the corresponding properties of the usual Bessel functions. See \cite{AShandbook,NIST} for more on Bessel functions.
\subsection{First Kind}

The ultraspherical Bessel equation
\[
z^2 w''+(d-1)zw'+(z^2-\ell(\ell+d-2))w=0
\]
is solved by $j_\ell(z)=z^{-s}J_{s+\ell}(z)$, where $s=(d-2)/2$, and $J_{\nu}$ is the Bessel function of the first kind. In series form, we can thus write
\[
j_\ell(z)=\sum_{k=0}^\infty\frac{(-1)^k 2^{1-d/2-2k-\ell}}{k!\Gamma(k+d/2+\ell)}z^{2k+\ell}=:\sum_{k=0}^\infty(-1)^kc_{d,l}(k)z^{2k+\ell}.
\]
The modified ultraspherical Bessel equation
\[
z^2 w''+(d-1)zw'-(z^2+\ell(\ell+d-2))w=0
\]
is solved by $i_\ell(z)=z^{-s}I_{s+\ell}(z)$, where $I_{\nu}$ is the modified Bessel function of the first kind. The series expansion for $i_{\ell}$ is the same as $j_\ell(z)$ with the alternating term removed. The ultraspherical Bessel functions $j_{\ell},i_{\ell}$ satisfy standard recurrence relations (see p.361 and p.376 of \cite{AShandbook} for the analogous properties of $J_{\nu}$ and $I_{\nu}$):

\begin{align}
\frac{d-2+2\ell}{z}j_\ell(z)&=j_{\ell-1}(z)+j_{\ell+1}(z), \label{jrecur1}\\
j_\ell'(z)&=\frac{\ell}{z}j_\ell(z)-j_{\ell+1}(z), \label{jrecur2}\\
&=j_{\ell-1}(z)-\frac{\ell+d-2}{z}j_\ell(z), \label{jrecur3}
\end{align}
\begin{align}
\frac{d-2+2\ell}{z}i_\ell(z)&=i_{\ell-1}(z)-i_{\ell+1}(z), \label{irecur1}\\
i_\ell'(z)&=\frac{\ell}{z}i_\ell(z)+i_{\ell+1}(z), \label{irecur2}\\
&=i_{\ell-1}(z)-\frac{\ell+d-2}{z}i_\ell(z). \label{irecur3}
\end{align}

Since the ultraspherical Bessels $j_1(z)$ and $i_1(z)$ are of particular interest, we introduce coefficients $c_k$ where
\[
j_1(z)=\sum_{k=0}^\infty (-1)^k c_k z^{2k+1},\quad i_1(z)=\sum_{k=0}^\infty c_k z^{2k+1}.
\]

Roots of Bessel functions and their derivatives are important in the study of eigenvalues of disks and balls. Let $p_{\ell,k}$ denote the $k$th positive zero of $j_{\ell}'(z)$. The following lemma by Lorch and Szeg\H{o} provides useful bounds on these zeros.

\begin{lem}[Lorch and Szeg\H{o}, \cite{LorchSzego}]\label{lem:LS} For $d\geq 2, l\geq 1$, we have
\[
\frac{\ell(d+2\ell)(d+2\ell+2)}{d+4l+2}<(p_{\ell,1})^2<l(d+2\ell).
\]
For $d\geq 2$, we therefore have
\[
d<(p_{1,1})^2<d+2.
\]
\end{lem}

We also collect specific properties of low-order Bessel $j$ functions. Most of these results were previously established in \cite{chasmanball}.

\begin{lem}[Properties of Bessel $j$ Functions]\label{lem:besselprops} We have the following properties:
\begin{enumerate}
\item $j_1''(z)<0$ on $(0,p_{1,1}]$.
\item $j_1(z), j_2(z), j_3(z)>0$ and $j_1^{(4)}(z)>0$ on $(0,p_{1,1}]$.
\item $j_1'(z)>0$ on $(0,p_{1,1})$.
\item For $\ell>0$, $zj_\ell'(z)/j_\ell(z)$ is positive and strictly decreasing on $(0,p_{\ell,1})$.
\item Let $d_k$ be the series coefficients of $i_1''(z)$. That is,
\[
d_k=\frac{(2k+1)2^{1-2k-d/2}}{(k-1)!\Gamma(k+1+d/2)}.
\]
Then we have
\begin{align*}
-d_1z+d_2z^3&\geq j_1''(z)  &&\textup{when }z\in[0,\sqrt{3(d+2)/(d+5)}],\\
d_1z+\frac{6}{5}d_2z^3&\geq i_1''(z) &&\textup{when }z\in[0,\sqrt{3}].
\end{align*}
\item For all $a\in(0,p_{1,1}]$, $b>0$, and $r\in(0,1]$, we have $a/b> j_1(ar)/i_1(br)$, with equality in the limit as $r\to0^+$. 
\end{enumerate}
\end{lem}

\begin{proof}
Properties 1, the second half of 2, 3, and 5 are proved in \cite{chasmanball}; see Lemmas 8, 9, 6, and 10, respectively.
 
To prove the first half of property 2, recall that $j_{\nu,k}$ denotes the $k$th positive zero of $J_{\nu}$. Since $J_{s+1}$ is positive on $(0,j_{s+1,1})$ and vanishes at the endpoints, the same is true of $j_1$. It follows that $0<p_{1,1}<j_{s+1,1}$, and so $j_1>0$ on $(0,p_{1,1}]$. By standard properties of zeros of Bessel functions, $j_{s+1,1}<j_{s+2,1}$, and so we deduce that that $j_2>0$  and $j_3>0$ on $(0,p_{1,1}]$ as well.
 
To prove property 4, we first note that $zj_\ell'(z)/j_\ell(z)$ vanishes at $z=p_{\ell,1}$. Recalling that $s=(d-2)/2$ and using the definition of the ultraspherical Bessel functions, we rewrite
\begin{align*}
\frac{zj_\ell'(z)}{j_\ell(z)}&=z\frac{d}{dz}\log(j_\ell(z))\\
&=z\frac{d}{dz}\left(-s\log(z)+\log(J_{s+\ell}(z))\right)\\
&=-s+z\frac{d}{dz}\log(J_{s+\ell}(z))\\
&=-s+\frac{zJ_{s+\ell}'(z)}{J_{s+\ell}(z)}.
\end{align*}

Since $j_{s+\ell,1}$ denotes the first positive root of $J_{s+\ell}$, it is also the first positive root of $j_\ell$, and so $p_{\ell,1}<j_{\ell+s,1}$. By Lemma 11 of \cite{FreitasLaugesen1}, the function $zJ_\nu'(z)/J_\nu(z)$ strictly decreases from $\nu$ to $-\infty$ on $(0,j_{\nu,1})$. Thus $zj_\ell'(z)/j_\ell(z)$ strictly decreases from $\ell$ to $-\infty$ on $(0,j_{s+\ell,1})$, and hence strictly decreases from $\ell$ to $0$ on $(0,p_{\ell,1}]$, as desired.

We finally establish property 6. The quotient $j_1(ar)/i_1(br)$ is strictly decreasing in $r$ for $r\in(0,1]$, since by relations \eqref{jrecur2} and \eqref{irecur2}
\begin{align*}
\frac{d}{dr}\frac{j_1(ar)}{i_1(br)}&=\frac{ai_1(br)j_1'(ar)-bj_1(ar)i_1'(br)}{i_1(br)^2}\\
&=\frac{i_1(br)(j_1(ar)-arj_2(ar))-j_1(ar)(i_1(br)+bri_2(br))}{ri_1(br)^2}\\
&=-\frac{ai_1(br)j_2(ar)+bj_1(ar)i_2(br)}{i_1(br)^2},
\end{align*}
which is negative for $r\in(0,1]$, $a\in(0,p_{1,1}]$. Using the series expansions, we see that as $r\to0^+$,
\[
\frac{j_1(ar)}{i_1(br)}=\frac{c_0 ar+\BigO(r^3)}{c_0 br+\BigO(r^3)}=\frac{c_0 a+\BigO(r^2)}{c_0 b+\BigO(r^2)}\to \frac{a}{b}.\qedhere
\]
\end{proof}

\subsection{Second Kind}
Bessel and modified Bessel functions of the second kind are singular at the origin. We develop their theory here only for the sake of completeness; such functions to not appear in the unit ball eigenfunctions of \eqref {eqn:robinplate}. We deviate from standard notation, using $N_\nu$ instead of the standard $Y_\nu$ for the Bessel function of the second kind, to avoid confusion with spherical harmonics.

For the Bessel equation, we choose Weber functions as follows (see p.358 of  \cite{AShandbook}):
\begin{equation}\label{eq:Nnonint}
N_\nu(z)=\frac{J_{\nu}(z)\cos(\nu \pi)-J_{-\nu}(z)}{\sin(\nu\pi)} \qquad \text{for noninteger indexes $\nu$.}\\
\end{equation}
The ultraspherical Bessel equation is then solved by $n_\ell(z)=z^{-s}N_{s+\ell}(z)$.

For the modified Bessel equation, the second standard solution $K_\nu(z)$ is defined by (see p.375 of  \cite{AShandbook})
\begin{equation}\label{eq:Knonint}
K_\nu(z)=\frac{\pi}{2}\frac{I_{-\nu}(z)-I_{\nu}(z)}{\sin(\nu\pi)} \qquad \text{for noninteger indexes $\nu$.}
\end{equation}
The ultraspherical modified Bessel equation is then solved by $k_\ell(z)=z^{-s}K_{s+\ell}(z)$. Recall $s=(d-2)/2$; this is an integer for even dimensions and a noninteger for odd dimensions.

For integer indexes, we have the following series solutions (see p.360 and p.375 of \cite{AShandbook}):
\begin{align}
N_n(z)&=-\frac{1}{\pi}\sum_{k=0}^{n-1}\frac{(n-k-1)!}{k!2^{2k-n}}z^{2k-n}+\frac{2}{\pi}\ln(z/2)J_n(z)\nonumber \\
&\qquad-\frac{1}{\pi}\sum_{k=0}^\infty\frac{\psi(k+1)+\psi(n+k+1)}{k!(n+k)!2^{2k+n}}(-1)^kz^{2k+n},\label{eq:Ndef}\\
K_n(z)&=\sum_{k=0}^{n-1}\frac{(n-k-1)!}{k!2^{2k-n+1}}(-1)^kz^{2k-n} +(-1)^{n+1}\ln(z/2)I_n(z)\nonumber\\
&\qquad+(-1)^n\sum_{k=0}^\infty \frac{\psi(k+1)+\psi(n+k+1)}{k!(n+k)!2^{2k+n+1}}z^{2k+n}, \label{eq:Kdef}
\end{align}
where $\psi$ denotes the digamma function.

\section{Eigenfunctions of the Unit Ball}\label{sect:UBall}
Throughout this section, we consider problem \eqref{eqn:robinplate} with $\tau>0$ when $\Omega$ equals the unit ball $\mathbb{B}\subset \mathbb{R}^d$:

\begin{equation}\label{eqn:robinplateball}
\begin{cases}
\Delta^2 u-\tau\Delta u=\Lambda u &\text{in $\mathbb{B}$},\\
Mu=\frac{\partial^2u}{\partial n^2}=0  &\text{on $\partial \mathbb{B}$},\\
Vu=\tau\frac{\partial u}{\partial n}-\frac{\partial \Delta u}{\partial n}-\mathrm{div}_{\partial\Omega}[P_{\partial\Omega}((D^2u)n)]+\alpha u=0 &\text{on $\partial \mathbb{B}$.}
\end{cases}
\end{equation}
The boundary conditions in \eqref{eqn:robinplateball} become especially simple in this setting. Using spherical coordinates, $\partial u/\partial n= u_r$ and $\Delta =\partial_r^2+\left((d-1)/r\right) \partial_r+(1/r^2)\Delta_S$, where $\Delta_S$ is the angular part of the Laplacian (the spherical Laplacian). In this case, $\Delta_{\partial\mathbb{B}}=\Delta_S$. We thus have (see Proposition 6 of \cite{chasmanineq})
\begin{align*}
Mu&=u_{rr}=0 &&\text{when $r=1$,} \\
Vu&=\tau u_r-(\Delta u)_r-\Delta_S\left(u_r-\frac{u}{r}\right)+\alpha u =0&&\text{when $r=1$.}
\end{align*}
The proof of Proposition 1 in \cite{chasmanball} shows that we can choose an eigenbasis of functions $u$ which satisfy the eigenvalue equation of \eqref{eqn:robinplateball} with the form
\begin{equation}\label{eq:posfactor}
u(r,\hat\theta)=R(r)Y_\ell(\hat\theta),\qquad r\in[0,1),\quad \hat \theta \in \partial \mathbb{B},
\end{equation}
with $Y_\ell$ an $\ell$th-order spherical harmonic.  If we define $F_\ell=\ell(\ell+d-2)$, the boundary conditions $Mu=0$ and $Vu=0$ reduce to the radial boundary conditions
\begin{align}
M_{\textup{rad}}R&=R''=0 &&\text{when $r=1$,}\label{eq:Mrad}\\
V_{\textup{rad}}R&=\tau R'-\left(R''+\frac{d-1}{r}R'-\frac{F_\ell}{r^2}R\right)'+F_\ell(R'-R)+\alpha R=0&&\text{when $r=1$.}\label{eq:Vrad}
\end{align}

\subsection{Positive Eigenvalues}\label{sect:poseval}
Our treatment here is identical to that of the free plate. Assume that $u$ is an eigenfunction of \eqref{eqn:robinplateball} with corresponding eigenvalue $\Lambda>0$. We factor the eigenvalue equation as
\[
(\Delta+a^2)(\Delta-b^2)u=0, \qquad \tau=b^2-a^2,\quad \Lambda=a^2b^2.
\]
There are nonnegative integers $\ell_1,\ell_2$ and constants $A,B,C,D$ where
\[
u=\left(Aj_{\ell_1}(ar)+Cn_{\ell_1}(ar)\right)Y_{\ell_1}(\hat \theta)+\left(Bi_{\ell_2}(br)+Dk_{\ell_2}(br)\right)Y_{\ell_2}(\hat \theta), \qquad r\in[0,1),\quad \hat \theta \in \partial \mathbb{B}.
\]
From \eqref{eq:posfactor}, we may assume the indices $\ell_1=\ell_2=\ell$ are equal. By Proposition \ref{prop:SpecProps}, eigenfunctions of \eqref{eqn:robinplateball} are smooth on $\overline{\mathbb{B}}$. Since the radial part $R$ of $u$ must be smooth on $[0,\infty)$, vanish at the origin when $\ell=1$, and have a vanishing derivative at the origin when $\ell=0$, a case by case analysis using \eqref{eq:Nnonint}, \eqref{eq:Knonint}, \eqref{eq:Ndef}, and \eqref{eq:Kdef} shows that $R$ must be a linear combination of ultraspherical Bessel and modified Bessel functions of the first kind alone. That is
\[
u=Aj_{\ell}(ar)Y_{\ell}(\hat \theta)+Bi_{\ell}(br)Y_{\ell}(\hat \theta).
\]
Since $u$ is an eigenfunction, it cannot be the case that $A$ and $B$ simultaneously vanish. The following determinant therefore vanishes: 
\begin{align}
W_\ell(a)&=\det  \begin{bmatrix} \Mrad j_\ell(ar) & \Mrad i_\ell(br) \\ \Vrad j_\ell(ar) & \Vrad i_\ell(br) \end{bmatrix} \bigg|_{r=1}\label{eq:Wldef}\\
&=\Mrad j_\ell(ar)\Vrad i_\ell(br)-\Mrad i_\ell(br)\Vrad j_\ell(ar) \bigg|_{r=1}\nonumber \\
&=a^2j_\ell''(a)(-a^2 bi_\ell'(b)+F_\ell(bi_\ell'(b)-i_\ell(b))+\alpha i_\ell(b))\nonumber \\
&\qquad-b^2i_\ell''(b)(b^2aj_\ell'(a)+F_\ell(aj_\ell'(a)-j_\ell(a))+\alpha j_\ell(a))\nonumber \\
&=0 \nonumber.
\end{align}

Conversely, if $W_{\ell}(a)=0$ for $a>0$ and $b^2=\tau+a^2$, then by definition, there exist constants $A, B$ not both zero such that $u=\left(Aj_{\ell}(ar)+Bi_{\ell}(br)\right)Y_{\ell}(\hat \theta)$ satisfies the boundary conditions in \eqref{eqn:robinplateball}. Moreover, since $\Delta (j_{\ell}(ar)Y_{\ell}(\hat \theta))=-a^2j_{\ell}(ar)Y_{\ell}(\hat \theta)$ and $\Delta (i_{\ell}(br)Y_{\ell}(\hat \theta))=b^2i_{\ell}(br)Y_{\ell}(\hat \theta)$, $u$ satisfies the eigenvalue equation in \eqref{eqn:robinplateball} with eigenvalue $\Lambda=a^2b^2$. We conclude that the eigenvalues of  \eqref{eqn:robinplateball} are precisely determined by positive solutions to the equation $W_{\ell}(a)=0$ and are increasing in $a$.

\subsection{Zero Eigenvalues}
By inspecting the Rayleigh quotient, our assumption that $\tau>0$ means that zero eigenvalues only occur when $\alpha\leq 0$. The eigenvalue equation now factors as
\[
\Delta(\Delta-b^2)u=0, \qquad \tau=b^2.
\]
Our solutions now take the form
\[
u=\left(Ar^{\ell}+Cr^{-(\ell+d-2)}\right)Y_{\ell}(\hat \theta)+\left(Bi_{\ell}(br)+Dk_{\ell}(br)\right)Y_{\ell}(\hat \theta), \qquad r\in[0,1),\quad \hat \theta \in \partial \mathbb{B},
\]
except when $d=2$ and $l=0$, where the singular solution $r^{-(\ell+d-2)}$ is replaced by $\ln(r)$. As in the positive eigenvalue case, the radial part $R$ of $u$ is smooth on $[0,\infty)$, vanishes at the origin when $\ell=1$, and has a vanishing derivative at the origin when $\ell=0$. Again, a case by case analysis shows that $R$ must be a linear combination of $r^{\ell}$ and an ultraspherical modified Bessel function of the first kind. Thus,
\begin{equation}\label{eq:efctnzero}
u=\left(Ar^{\ell}+Bi_{\ell}(br)\right)Y_{\ell}(\hat \theta).
\end{equation}

\subsection{Negative Eigenvalues}

From the Rayleigh quotient, the assumption $\tau>0$ means that negative eigenvalues only occur when $\alpha<0$. In the negative eigenvalue regime, the factorization of the eigenfunction equation depends on the relationship between $\tau$ and $\Lambda$.

When $\Lambda\neq-\tau^2/4$, we factor the eigenvalue equation as
\[
(\Delta-a^2)(\Delta-b^2)u=0, \qquad a\neq b, \tau=a^2+b^2, \Lambda=-a^2b^2.
\]
Here, the radial part of $u$ is a linear combinations of $i_\ell(ar)$, $i_\ell(br)$, $k_\ell(ar)$, and $k_\ell(br)$. But arguing as in the positive and zero eigenvalue case, we dismiss the Bessel $k_{\ell}$'s, and so
\begin{equation}\label{eq:efctnlneg}
u=\left(Ai_{\ell}(ar)+Bi_{\ell}(br)\right)Y_{\ell}(\hat \theta), \qquad r\in[0,1),\quad \hat \theta \in \partial \mathbb{B}.
\end{equation}

If $\Lambda=-\tau^2/4$, the factorization has a repeated factor:
\[
(\Delta-b^2)^2u=0, \qquad \tau=2a^2, \Lambda=-b^4.
\]
However, this ``isolated'' case can be ignored, since the following lemma says that there is at most one $\alpha$ value where $\Lambda=-\tau^2/4$.
\begin{lem}\label{lem:L2sinc}
For $\alpha\leq 0$ and $\tau>0$, the first eigenvalue $\Lambda_1(\mathbb{B}; \tau, \alpha)$ is strictly increasing in $\alpha$.
\end{lem}

\begin{proof}
Say $\alpha_1< \alpha_2\leq 0$. As observed in the proof of Proposition \ref{Prop:Cont}, $\Lambda_1(\mathbb{B}; \tau, \alpha_1) \leq \Lambda_1(\mathbb{B}; \tau, \alpha_2)$. Suppose that $\Lambda_1(\mathbb{B}; \tau, \alpha_1) = \Lambda_1(\mathbb{B}; \tau, \alpha_2)$ with $u_{\alpha_2}$ an eigenfunction for $\Lambda_1(\mathbb{B}; \tau, \alpha_2)$. Then the Rayleigh quotient \eqref{eq:RQ} satisfies
\[
\Lambda_1(\mathbb{B}; \tau, \alpha_2)=\Lambda_1(\mathbb{B}; \tau, \alpha_1)\leq Q[\mathbb{B}; \alpha_1, u_{\alpha_2}]   \leq Q[\mathbb{B}; \alpha_2, u_{\alpha_2}] = \Lambda_1(\mathbb{B}; \tau, \alpha_2),
\]
and so $Q[\mathbb{B}; \alpha_1, u_{\alpha_2}]=\Lambda_1(\mathbb{B}; \tau, \alpha_1)$. That is, $u_{\alpha_2}$ solves problem \eqref{eqn:robinplateball} for $\alpha=\alpha_1$. Writing $R_{\alpha_2}$ for the radial part of $u_{\alpha_2}$ (we later show that, in fact, $u_{\alpha_2}$ is radial, but that is not needed here), it follows that $R_{\alpha_2}$ satisfies the equations
\begin{align*}
\tau R_{\alpha_2}'-\left(R_{\alpha_2}''+\frac{d-1}{r}R_{\alpha_2}'-\frac{d-1}{r^2}R_{\alpha_2}\right)'+(d-1)(R_{\alpha_2}'-R_{\alpha_2})+\alpha_1 R_{\alpha_2}=0&&\text{when $r=1$},\\
\tau R_{\alpha_2}'-\left(R_{\alpha_2}''+\frac{d-1}{r}R_{\alpha_2}'-\frac{d-1}{r^2}R_{\alpha_2}\right)'+(d-1)(R_{\alpha_2}'-R_{\alpha_2})+\alpha_2 R_{\alpha_2}=0&&\text{when $r=1$.}
\end{align*}
Taking the difference of these equations shows that $(\alpha_1-\alpha_2)R_{\alpha_2}(1)=0$, and so $R_{\alpha_2}(1)=0$. That is, $u_{\alpha_2}$ is an eigenfunction for problem \eqref{eqn:robinplateball} for $\alpha=0$. And since $\Lambda_1(\mathbb{B}; \tau,\alpha_2)\leq \Lambda_1(\mathbb{B}; \tau, 0)=0$, we deduce that $\Lambda_1(\mathbb{B}; \tau,\alpha_2)=0$ and $u_{\alpha_2}$ is constant. But then $R_{\alpha_2}(1)=0$ implies $u_{\alpha_2}$ vanishes everywhere, so cannot be an eigenfunction. We conclude  $\Lambda_1(\mathbb{B}; \tau, \alpha_1) < \Lambda_1(\mathbb{B}; \tau, \alpha_2)$ as desired.
\end{proof}

\section{Identifying Low Eigenvalues of the Unit Ball}

Throughout this section, we continue to focus on problem \eqref{eqn:robinplateball} with an eye towards identifying the form of eigenfunctions for low eigenvalues. Throughout, we continue to assume $\tau>0$. We begin with the following preliminary lemma.

\begin{lem}\label{lem:incrRQ} For all $\alpha \in \mathbb{R}$, the lowest eigenvalue $\Lambda_1(\mathbb{B};\tau,\alpha)$ of problem \eqref{eqn:robinplateball} corresponds to either $\ell=0$ or $\ell=1$.
\end{lem}

\begin{proof}
We examine the Rayleigh quotient with (nonzero) trial functions of the form $u=R(r)Y_\ell(\hat \theta)$ for fixed $R(r)$ to identify the modes that correspond to the lowest eigenvalue. Our spherical harmonics $Y_\ell$ are normalized so that
\[
\int_{\partial\BB}Y_\ell^2\,dS=1\qquad\text{and}\qquad\int_{\partial\BB}|\nabla_S Y_\ell|^2\,dS=\ell(\ell+d-2)=F_\ell.
\]

The Rayleigh quotient for $u$ is then
\begin{equation}\label{eq:RQsep}
Q[u]=\frac{\int_\mathbb{B} \left( |D^2u|^2+\tau |Du|^2\right)\,dx+\alpha\int_{\partial\mathbb{B}}u^2\,dS}{\int_\mathbb{B} u^2\,dx}.
\end{equation}
Note that the second term in the numerator of \eqref{eq:RQsep} equals $\alpha R(1)^2$, while the denominator equals $\int_0^1R(r)^2r^{d-1}\,dr$. Both of these terms are independent of $F_{\ell}$. The proof of Theorem 3 in \cite{chasmanball} shows that
\begin{align}
\int_{\mathbb{B}}(|D^2u|^2+\tau|Du|^2)\,dx&=\int_0^1\left(\frac{2F_\ell}{r^4}\left(rR'-\frac{3}{2}R\right)^2+\frac{F_\ell(F_\ell-d-1/2)}{r^4}R^2+\tau \frac{F_\ell}{r^2}R^2\right)r^{d-1}\,dr \label{eq:MCcomp}\\
&\qquad+\int_0^1\left((R'')^2+\frac{d-1}{r^2}(R')^2+\tau(R')^2\right)r^{d-1}\,dr. \nonumber
\end{align}

Since $\tau>0$, \eqref{eq:MCcomp} is strictly increasing in $F_\ell$ when $2F_\ell\geq d+1/2$, and hence strictly increasing in $\ell$ for $\ell\geq 2$. However, $F_\ell(F_\ell-d-1/2)$ equals $-3(d-1)/2<0$ when $\ell=1$, and equals  $2d(d-1/2)>0$ when $\ell=2$. Thus, \eqref{eq:MCcomp} strictly increases in $\ell$ for $\ell \geq 1$. We deduce that the Rayleigh quotient \eqref{eq:RQsep} strictly increases in $\ell$ when $\ell\geq 1$. Since
\[
\Lambda_1(\mathbb{B};\tau,\alpha)=\inf \{Q[u]:u=R(r)Y_\ell(\hat \theta), R\in C^{\infty}[0,1], \ell\geq 0\},
\]
we see that the lowest eigenvalue for the unit ball always corresponds to $\ell=0$ or $\ell=1$, giving either a purely radial mode or one with simple angular dependence.
\end{proof}

\begin{lem}\label{lemL1rad}
For $\alpha \in [-\tau,0)$, all eigenfunctions for $\Lambda_1(\mathbb{B}; \tau, \alpha)$ correspond to $\ell=0$.
\end{lem}

\begin{proof}
The result follows from Lemma \ref{lem:incrRQ}, after we show that no eigenfunction corresponds to $\ell=1$. So suppose $\Lambda_1(\Omega; \tau, \alpha)$ has an eigenfunction with $\ell=1$. Then the eigenspace for $\Lambda_1(\mathbb{B}; \tau, \alpha)$ has dimension at least $2$. Thus $\Lambda_1(\mathbb{B}; \tau, \alpha)=\Lambda_2(\mathbb{B}; \tau, \alpha)<0$ by Lemma \ref{lem:L2sinc} since $\Lambda_1(\mathbb{B}; \tau, 0)=0$. By Proposition \ref{Prop:Cont}, there exists $\alpha' \in (\alpha,0)$ where $\Lambda_2(\mathbb{B}; \tau, \alpha')=0$. Then all eigenfunctions for $\Lambda_1(\mathbb{B}; \tau, \alpha')$ must correspond to $\ell=0$, since otherwise $\Lambda_1(\mathbb{B}; \tau, \alpha')=\Lambda_2(\mathbb{B}; \tau, \alpha')$ would be negative. The proof of Lemma \ref{lem:incrRQ} shows that for nonzero $R\in C^{\infty}[0,1]$ fixed, the Rayleigh quotient $Q[\mathbb{B}; \alpha', RY_{\ell}]$ is strictly increasing in $\ell$ when $\ell\geq 1$. Since eigenfunctions for $\Lambda_1(\mathbb{B}; \tau, \alpha')$ correspond to $\ell=0$, such eigenfunctions are orthogonal to all functions of the form $RY_{\ell}$ for $\ell \geq 1$. From the variational formula, it follows that eigenfunctions for $\Lambda_2(\mathbb{B}; \tau, \alpha')$ correspond to either $\ell=0$ or $\ell=1$.

When $\ell=0$, \eqref{eq:efctnzero} shows that an eigenfunction for $\Lambda_2(\mathbb{B}; \tau, \alpha')=0$ is a linear combination of $1$ and $i_0(\sqrt{\tau}r)$. The radial boundary conditions \eqref{eq:Mrad} and \eqref{eq:Vrad} become
\begin{align*}
\Mrad 1\big|_{r=1}&=0,\\
\Mrad i_0(\sqrt{\tau}r)\big|_{r=1}&=\tau i_0''(\sqrt{\tau}), \\
\Vrad 1\big|_{r=1}&=\alpha',\\
\Vrad i_0(\sqrt{\tau}r)\big|_{r=1}&=\alpha' i_0(\sqrt{\tau}).
\end{align*}
Since the  linear combination of $1$ and $i_0(\sqrt{\tau}r)$ is nontrivial, it must be that
\[
\Mrad 1 \Vrad i_0(\sqrt{\tau}r)-\Mrad i_0(\sqrt{\tau}r)\Vrad 1\bigg|_{r=1}= -\tau \alpha' i_0''(\sqrt{\tau})=0.
\]
Since $i_0''(z)$ is never zero, this is only possible when $\alpha'=0$.

Next say $\ell=1$. In this case, the radial part of an eigenfunction for $\Lambda_2(\mathbb{B}; \tau, \alpha')=0$ is a linear combination of $r$ and $i_1(\sqrt{\tau}r)$. The radial boundary conditions applied to each function yields
\begin{align*}
\Mrad r\big|_{r=1}&=0,\\
\Mrad i_1(\sqrt{\tau}r)\big|_{r=1}&=\tau i_1''(\sqrt{\tau}), \\
\Vrad r\big|_{r=1}&=\alpha'+\tau,\\
\Vrad i_1(\sqrt{\tau}r)\big|_{r=1}&=(d-1)(\sqrt{\tau}i_1'(\sqrt{\tau})-i_1(\sqrt{\tau}))+\alpha' i_1(\sqrt{\tau}).
\end{align*}
Again, since the linear combination of  $r$ and $i_1(\sqrt{\tau}r)$ is nontrivial, we must have
\[
\Mrad r \Vrad i_1(\sqrt{\tau}r)-\Mrad i_1(\sqrt{\tau}r)\Vrad r\bigg|_{r=1}= -(\alpha'+\tau)\tau i_0''(\sqrt{\tau})=0.
\]
This last equation only holds when $\alpha'=-\tau$. We deduce $\Lambda_2(\mathbb{B}; \tau, \alpha')$ must be nonzero, and this contradiction shows that all eigenfunctions for $\Lambda_1(\mathbb{B}; \tau, \alpha)$ correspond to $\ell=0$.
\end{proof}
We now turn our attention to the structure of eigenfunctions for $\Lambda_2(\Omega; \tau, \alpha)$.

\begin{lem}\label{lem:2ndl01}
For $\alpha \in (-\tau,0)$, the second eigenvalue $\Lambda_2(\mathbb{B}; \tau, \alpha)$ of problem \eqref{eqn:robinplateball}  is nonzero and corresponds to either $\ell=0$ or $\ell=1$.
\end{lem}

\begin{proof}
Lemma \ref{lemL1rad} gives that eigenfunctions for $\Lambda_1(\mathbb{B}; \tau, \alpha)$ correspond to $\ell=0$. The argument in the proof of Lemma \ref{lemL1rad}, with $\Lambda_2(\mathbb{B}; \tau, \alpha)$ in place of $\Lambda_2(\mathbb{B}; \tau, \alpha')$, shows that eigenfunctions correspond to $\ell=0$ or $\ell=1$, and that if $\Lambda_2(\mathbb{B}; \tau, \alpha)=0$, then $\alpha=0$ or $\alpha=-\tau$. We conclude $\Lambda_2(\mathbb{B}; \tau, \alpha)$ is nonzero.
\end{proof}

We next refine the previous lemma.

\begin{lem}\label{lem:Wl01}
For $\alpha \in (-\tau,0)$, the second eigenvalue $\Lambda_2(\mathbb{B}; \tau, \alpha)$ of problem \eqref{eqn:robinplateball}  is positive and corresponds to $\ell=1$.
\end{lem}

\begin{proof}

From Lemma \ref{lem:2ndl01}, $\Lambda_2(\mathbb{B}; \tau, \alpha)$ is nonzero. By Proposition \ref{Prop:Cont}, $\Lambda_2(\mathbb{B}; \tau, \alpha)$ is continuous in $\alpha$, and since $\Lambda_2(\mathbb{B}; \tau, 0)>0$, it must be the case that $\Lambda_2(\mathbb{B}; \tau, \alpha)>0$ for $\alpha \in (-\tau,0)$.

The work in Section \ref{sect:poseval} shows that eigenfunctions for $\Lambda_2(\mathbb{B}; \tau, \alpha)$ take the form $u=R(r)Y_{\ell}(\hat \theta)$, where
\[
R(r)=Aj_\ell(ar)+Bi_{\ell}(br).
\]
The positive constants $a$ and $b$ are determined by the equations $\tau=b^2-a^2$ and $\Lambda=a^2b^2$.

We examine the radial boundary operators applied to the ultraspherical Bessel functions:
\begin{align*}
\Mrad j_\ell(ar)\big|_{r=1}&=a^2j_\ell''(a),\\
\Mrad i_\ell(br)\big|_{r=1}&=b^2i_\ell''(b),\\
\Vrad j_\ell(ar)\big|_{r=1}&=b^2aj_\ell'(a)+F_\ell(aj_\ell'(a)-j_\ell(a))+\alpha j_\ell(a),\\
\Vrad i_\ell(br)\big|_{r=1}&=-a^2bi_\ell'(b)+F_\ell(bi_\ell'(b)-i_\ell(b))+\alpha i_\ell(b).
\end{align*}
Recall that $p_{1,1}$ denotes the first positive zero of $j_1'(z)$. In light of Lemma \ref{lem:2ndl01}, we show that $W_1(a)$, defined in \eqref{eq:Wldef}, has a zero in $(0,p_{1,1})$, but $W_0(a)$ is positive on $(0,p_{1,1}]$.

When $\ell=0$, we have $F_0=0$. Using the recurrence relations \eqref{jrecur2} and \eqref{irecur2}, 
\begin{align}
\Mrad j_0(ar)\big|_{r=1}&=a^2j_0''(a)=-a^2j_1'(a), \label{eq:Mradj0}\\
\Mrad i_0(br)\big|_{r=1}&=b^2i_0''(b)=b^2i_1'(b), \label{eq:Mradi0}\\
\Vrad j_0(ar)\big|_{r=1}&=b^2aj_0'(a)+\alpha j_0(a)=-b^2aj_1(a)+\alpha j_0(a), \label{eq:Vradj0}\\
\Vrad i_0(br)\big|_{r=1}&=-a^2bi_0'(b)+\alpha i_0(b)=-a^2bi_1(b)+\alpha i_0(b). \label{eq:Vradi0}
\end{align}
Recall from Lemma~\ref{lem:besselprops} that $j_1$ is positive, $j_1' $ is nonnegative, and $j_1''$ is negative on $(0, p_{1,1}]$. By \eqref{jrecur3}, $j_0(z)=j_1'(z)+((d-1)/z)j_1(z)$, and so $j_0$ is positive on $(0,p_{1,1}$] as well. The modified Bessel functions $i_{\ell}$ and their derivatives are all positive on $(0,\infty)$. Recalling $\alpha<0$, we observe that \eqref{eq:Mradj0} is nonpositive, \eqref{eq:Mradi0} is positive, and \eqref{eq:Vradj0} and \eqref{eq:Vradi0} are negative on $(0,p_{1,1}]$. Thus,
\[
W_0(a)=\Mrad j_0(ar)\Vrad i_0(br)-\Mrad i_0(br)\Vrad j_0(ar) \big|_{r=1}
\]
is positive on this interval, and so the first positive eigenvalue corresponding to $\ell=0$ must occur with some $a>p_{1,1}$.

When $\ell=1$, we have
\begin{align}
\Mrad j_1(ar)\big|_{r=1}&=a^2j_1''(a), \label{eq:Mradj1}\\
\Mrad i_1(br)\big|_{r=1}&=b^2i_1''(b),\label{eq:Mradi1}\\
\Vrad j_1(ar)\big|_{r=1}&=b^2aj_1'(a)+(d-1)(aj_1'(a)-j_1(a))+\alpha j_1(a)\nonumber \\
&=(b^2+d-1)aj_1'(a)+(\alpha-(d-1))j_1(a),\label{eq:Vradj1}\\
\Vrad i_1(br)\big|_{r=1}&=-a^2bi_1'(b)+(d-1)(bi_1'(b)-i_1(b))+\alpha i_1(b)\nonumber \\
&=(d-1-a^2)bi_1'(b)+(\alpha-(d-1))i_1(b). \label{eq:Vradi1}
\end{align}

First, we evaluate $W_1(a)$ at $a=p_{1,1}$. In this case, \eqref{eq:Vradj1} simplifies to $(\alpha-(d-1))j_1(p_{1,1})<0$. Since $d\geq 2$, Lemma \ref{lem:LS} gives $p_{1,1}^2>d$, so $d-1-p_{1,1}^2$ is negative. Since the $b$-value corresponding to $a=p_{1,1}$ is positive, $i_1(b), i_1'(b)$, and $i_1''(b)$ are all positive. Thus \eqref{eq:Mradi1} is positive and \eqref{eq:Vradi1} is negative. Since $j_1''(p_{1,1})<0$, \eqref{eq:Mradj1} is negative. Thus
\[
W_1(p_{1,1})=\Mrad j_1(p_{1,1}r)\Vrad i_1(br)-\Mrad i_1(br)\Vrad j_1(p_{1,1}r) \big|_{r=1}>0.
\]

Next, we examine $W_1(a)$ as $a\to0^+$. We note that by the series expansion for $j_1$,
\[
j_1(a)=c_0 a-c_1a^3+\BigO(a^5), \qquad \textup{as }a\to0^+.
\]
Using \eqref{eq:Mradj1}, \eqref{eq:Vradj1}, and $b^2=a^2+\tau$, we have
\begin{align*}
\Mrad j_1(ar)\big|_{r=1}&=-6c_1a^3+\BigO(a^5),\\
\Vrad j_1(ar)\big|_{r=1}&=(\tau+\alpha)c_0 a+(c_0-(3\tau+\alpha+2(d-1))c_1)a^3+\BigO(a^5).
\end{align*}
The derivatives of $i_1$, and the corresponding radial boundary values, are a little more delicate to compute:
\begin{align*}
i_1(b)&=b\sum_{k=0}^\infty c_k (a^2+\tau)^k=b\left(c_0+\sum_{k=1}^\infty c_k \tau^k+a^2\sum_{k=1}^\infty kc_k \tau^{k-1}+\BigO(a^4)\right)\\
bi_1'(b)&=b\sum_{k=0}^\infty c_k(2k+1)(a^2+\tau)^k=b\left(c_0+\sum_{k=1}^\infty c_k(2k+1) \tau^k+a^2\sum_{k=1}^\infty k(2k+1)c_k \tau^{k-1}+\BigO(a^4)\right)\\
b^2i_1''(b)&=b\sum_{k=1}^\infty c_k(2k+1)(2k)(a^2+\tau)^k=b\left(\sum_{k=1}^\infty c_k(2k+1)(2k) \tau^k+a^2\sum_{k=1}^\infty 2k^2(2k+1)c_k \tau^{k-1}+\BigO(a^4)\right)\\
\Vrad i_1(br)\big|_{r=1}&=b\left(\sum_{k=0}^\infty c_k\tau^k(\alpha+2k(d-1))
+\BigO(a^2)\right).\\
\end{align*}
Combining the above computations,
\begin{align*}
W_1(a)&=b\left(-6c_1a^3+\BigO(a^5)\right)\left(\sum_{k=0}^\infty c_k\tau^k(\alpha+2k(d-1))
+\BigO(a^2)\right)\\
&\qquad-b\left((\tau+\alpha)c_0 a+\BigO(a^3)\right)\left(\sum_{k=1}^\infty c_k(2k+1)(2k) \tau^k+a^2\sum_{k=1}^\infty 2k^2(2k+1)c_k \tau^{k-1}+\BigO(a^4)\right)\\
&=ab\left(-c_0(\tau+\alpha)\sum_{k=1}^\infty c_k(2k+1)(2k)\tau^k+\BigO(a^2)\right).
\end{align*}
Recall the coefficients $c_k$ are all positive, $\tau+\alpha>0$, and $\tau>0$. Thus we see $W_1(a)<0$ as $a\to0^+$.  Since $W_1(a)$ is a continuous function, it must have a zero somewhere in the interval $(0,p_{1,1})$.
\end{proof}

The following theorem summarizes the results of this section.

\begin{thm}\label{thm:groundstate} For $\alpha\in[-\tau,0)$, the lowest eigenvalue $\Lambda_1(\mathbb{B}; \tau,\alpha)$ of problem \eqref{eqn:robinplateball} is negative and simple. Moreover, eigenfunctions corresponding to $\Lambda_1(\mathbb{B}; \tau,\alpha)$ are radial. The second eigenvalue $\Lambda_2(\mathbb{B}; \tau,\alpha)$ is nonnegative and equals zero only when $\alpha=-\tau$. For $\alpha \in (-\tau,0)$, corresponding eigenfunctions take the form $u_2=R(r) Y_1(\hat\theta)$, which has simple angular dependence. Furthermore, we may rescale so that
\[
R(r)=j_1(ar)+\gamma i_1(br),\qquad a\in(0,p_{1,1}), b^2=a^2+\tau, \Lambda_2(\mathbb{B}; \tau,\alpha)=a^2b^2.
\]
\end{thm}

\begin{proof}
Since $\Lambda_1(\mathbb{B}; \tau, 0)=0$, Lemma \ref{lem:L2sinc} gives that $\Lambda_1(\mathbb{B}; \tau, \alpha)<0$. We have $\Lambda_2(\mathbb{B}; \tau, \alpha)\geq0$ with equality only when $\alpha=-\tau$ by the proof of Lemma \ref{lemL1rad} and Lemma \ref{lem:Wl01}. Eigenfunctions for $\Lambda_1(\mathbb{B}; \tau, \alpha)$ are radial by Lemma \ref{lemL1rad}. The claims on the form of eigenfunctions for $\Lambda_2(\mathbb{B}; \tau,\alpha)$ follow from Lemma \ref{lem:Wl01} (and its proof) and the work of Section \ref{sect:poseval}. By rescaling, we can write
\[
R(r)=j_1(ar)+\gamma i_1(br),\qquad \gamma=-\frac{a^2j_1''(a)}{b^2i_1''(b)}.
\]
The constant $\gamma$ is determined by \eqref{eq:Mrad}. We are justified in rescaling $R$ so that the coefficient of $j_{1}$ equals 1 because $M_{\textup{rad}}i_{1}>0$ when $r=1$.
 
\end{proof}

\section{Constructing the Trial Functions}
We assume throughout this section that $\tau>0$ and $\alpha\in(-\tau,0)$. By Theorem~\ref{thm:groundstate}, the first eigenvalue $\Lambda_1(\mathbb{B}; \tau, \alpha)$ corresponds to $\ell=0$ and the second eigenvalue $\Lambda_2(\mathbb{B}; \tau, \alpha)$ corresponds to $\ell=1$. Following Weinberger's approach, we use the eigenfunctions for $\Lambda_2(\mathbb{B}; \tau, \alpha)$ to construct our trial functions. Let 
\[
R(r)=j_1(ar)+\gamma i_1(br), \qquad b^2=\tau+a^2, \gamma=-a^2j_1''(a)/b^2i_1''(b),
\]
be the radial part of an eigenfunction for $\Lambda_2(\mathbb{B}; \tau, \alpha)$ as in Theorem \ref{thm:groundstate}. Here $a$ is the first nonzero root of $W_1(a)$, and by Theorem \ref{thm:groundstate}, we have $a\in(0,p_{1,1})$.

 We then define $\rho$ to be the linear extension of $R$ as follows:
\begin{equation}\label{eq:pdef}
 \rho(r)=\begin{cases} R(r) &\textup{when }0\leq r\leq 1,\\
          R'(1)(r-1)+R(1) & \textup{when }r>1.
         \end{cases}
\end{equation}

\begin{lem}\label{lem:trialprops}
The function $\rho$ defined in \eqref{eq:pdef} satisfies the following properties:
\begin{enumerate}
 \item $\rho\geq 0$, $\rho'\geq 0$, and $\rho''\leq0$ on $[0,\infty)$. 
  \item $\rho-r\rho'$ is nonnegative and increasing on $[0,\infty)$.
  \item $\alpha\rho+\tau\rho'$ is decreasing on $[0,\infty)$ and positive on $[0,1]$.
\end{enumerate}
\end{lem}

\begin{proof}
\begin{enumerate}
\item We first consider these functions on the interval $[0,1]$. Since $a\in(0,p_{1,1})$ and $0<a< b$, the constant $\gamma=-a^2 j_1''(a)/b^2 i_1''(b)$ satisfies $0<\gamma\leq 1$ by Lemma \ref{lem:besselprops}.  Thus, $\rho(r)=j_1(ar)+\gamma i_1(br)$ and $\rho'(r)=aj_1'(ar)+\gamma b i_1'(br)$ are nonnegative on $[0,1]$ since all terms are nonnegative by Lemma \ref{lem:besselprops}.  For $\rho''$, note that by definition of $\gamma$, we have $\rho''(1)=0$ since
\[
\rho''(r)=\frac{a^2j_1''(ar)b^2i_1''(b)-a^2j_1''(a)b^2i_1''(br)}{b^2i_1''(b)}. 
\]
Also note $\rho''(0)=0$ since the Bessel functions have second derivatives that vanish at the origin. The fourth derivative of $\rho$ satisfies
\[
\rho^{(4)}(r)=a^4j_1^{(4)}(ar)+\gamma b^4 i_1^{(4)}(br).
\]
By Lemma \ref{lem:besselprops}, $\rho^{(4)}(r)\geq 0$ on $[0,1]$, and so $\rho''(r)$ is a convex function. Since $\rho''(0)=\rho''(1)=0$, we conclude $\rho''(r)\leq 0$ on $[0,1]$.

\quad On $[1,\infty)$, note that $\rho$ is linear, so $\rho''(r)=0$. We also have $\rho'(r)=\rho'(1)$, which by our work above is positive. Thus $\rho$ is increasing on $[1,\infty)$ with $\rho(1)> 0$, so $\rho$ is nonnegative on this interval.

\item Set $f(r)=\rho(r)-r\rho'(r)$. Then $f'(r)=-r\rho''(r)$, which is nonnegative on $[0,\infty)$ by part 1 of this lemma.  We also clearly have $f(0)=0$, so $f$ is nonnegative and increasing on $[0,\infty)$.

\item Now set $f(r)=\alpha\rho(r)+\tau\rho'(r)$. Note that $f'(r)=\alpha\rho'(r)+\tau\rho''(r)\leq0$ for all $r\geq0$ by part 1, so $f$ is decreasing. So to prove nonnegativity on $[0,1]$, it suffices to show $f(1)>0$.

\quad As we saw earlier, $\rho''(1)=0$. The boundary condition $\Vrad R\big|_{r=1}=0$ gives
\[
b^2aj_1'(a)-\gamma a^2bi_1'(b)+(d-1)(\rho'(1)-\rho(1))+\alpha \rho(1)=0.
\]
Thus,
\[
\alpha \rho(1) + \tau \rho'(1)=-a^3j_1'(a)+b^3\gamma i'_1(b)+(d-1)(\rho(1)-\rho'(1)).
\]
Writing $u=RY_1$ and using $\rho''(1)=0$, a direct computation shows that $\Delta u=(d-1)(\rho'-\rho)Y_1$ when $r=1$. Thus, using the recursion relations \eqref{jrecur2} and \eqref{irecur2}, we therefore see
\begin{align*}
\alpha \rho(1)+\tau \rho'(1)&=-a^3j_1'(a)+b^3\gamma i'_1(b)+a^2j_1(a)-b^2\gamma i_1(b)\\
&=-a^3\left(\frac{1}{a}j_1(a)-j_2(a)\right)+b^3\gamma\left(\frac{1}{b}i_1(b)+i_2(b)\right)+a^2j_1(a)-b^2\gamma i_1(b)\\
&=a^3j_2(a)+b^3\gamma i_2(b).
\end{align*}
This last expression is positive by Lemma \ref{lem:besselprops}.
\end{enumerate}
\end{proof}

We are now prepared to prove our center of mass result with trial functions of the form $u=\rho(r)Y_1(\hat\theta)$.

\begin{lem}\label{lem:COM}
With $\Omega$ as in Theorem \ref{thm:mainthm}, let $v$ denote an eigenfunction for $\Lambda_1(\Omega; \tau, \alpha)$ where $\alpha \in (-\tau,0)$. With $\rho$ defined in \eqref{eq:pdef}, define $u_k(x)=\rho(r)x_k/r$ for $k=1,\ldots,d$. If $\int_{\Omega}v\,dx\neq 0$, then by translating $\Omega$ suitably, we may assume
\[
\int_{\Omega}u_kv\,dx=0
\]
for $k=1,\ldots,d$.
\end{lem}

\begin{proof}
By rescaling, we assume $\int_{\Omega}v\,dx>0$. The following argument uses ideas from the proof of Proposition 2 in \cite{FreitasLaugesen1}. Define functions
\begin{align*}
G(r)&=\int_0^{r}\rho(t)\,dt,\quad r\in [0,\infty),\\
f(y)&=\int_{\Omega}G(|y-x|)v(x)\,dx, \quad y\in \mathbb{R}^d.
\end{align*}
Since $\rho(r)$ is continuous on $[0,\infty)$, $G$ is continuous on $[0,\infty)$. The function $f$ is continuous on $\mathbb{R}^d$ since $\Omega$ is bounded. Moreover, for $|y|$ large, the formula for $\rho$ in \eqref{eq:pdef} shows that the behavior of $f(y)$ is controlled by $\frac{1}{2}R'(1)|y|^2\int_{\Omega}v(x)\,dx$, which tends to $\infty$ as $|y|\to \infty$ since $R'(1)>0$. It follows that $f$ achieves a global minimum at some $y\in \mathbb{R}^d$, and at this point, $Df(y)=0$. Written out explicitly, for $k=1, \ldots, d$, we thus have
\begin{equation}
\frac{\partial f}{\partial y_k}(y)=\int_{\Omega}\rho(|y-z|)\frac{y_k-z_k}{|y-z|}v(z)\,dz=0.
\end{equation}
Making a change of variable $z=x+y$, we see
\[
\int_{\Omega-y}\rho(r)\frac{x_k}{r}v(x+y)\,dx=0,
\]
as desired.
\end{proof}

\section{Inequalities Needed for the Monotonicity of the Rayleigh Quotient}

In this section, we prove a number of technical inequalities needed to establish the monotonicity of our Rayleigh quotient. We begin by strengthening a lower bound on Robin eigenvalues found in \cite{FreitasLaugesen1}.

\begin{lem}[Lower Bound on the Second Robin Membrane Eigenvalue]\label{lem:RobinMembraneBound} Let $\alpha\in[-1,0]$. Then the second eigenvalue for the Robin membrane on the unit ball satisfies
\[
\lambda_2(\mathbb{B}; \alpha)\geq d(1+\alpha).
\]
\end{lem}
\begin{proof} By Propositions $4$ and $5$ of \cite{FreitasLaugesen1}, the eigenfunctions associated with the Robin membrane eigenvalues $\lambda_1(\mathbb{B}; \alpha)$ and $\lambda_2(\mathbb{B}; \alpha)$ can be written in the form
\[
u_1(r,\hat\theta)=R_1(r) \quad\text{and} \quad u_2(r,\hat\theta)=R_2(r)Y_1(\hat\theta).
\]
In particular, $Y_1=C x_k/r$ for some rectangular coordinate $x_k$ and constant $C$. Then all $H^1(\BB)$ functions of the form $R(r)Y_1(\hat\theta)$ are orthogonal to $u_1$, and hence are valid trial functions for the variational characterization of $\lambda_2(\mathbb{B}; \alpha)$:
\begin{align}
\lambda_2(\mathbb{B}; \alpha)&=\inf_{\substack{u\in H^1(\mathbb{B})\\ u\perp u_1}}\frac{\int_{\mathbb B} |Du|^2\,dx+\alpha \int_{\partial \mathbb B}u^2\,dS}{\int_{\mathbb B} u^2\,dx} \nonumber \\
&= \inf_{\substack{u\in H^1(\mathbb{B})\\ u=RY_1}}(\alpha+1)\frac{\int_{\mathbb B} |Du|^2\,dx}{\int_{\mathbb B} u^2\,dx}-\alpha \left( \frac{\int_{\mathbb B} |Du|^2\,dx -  \int_{\partial \mathbb B}u^2\,dS}{\int_{\mathbb B} u^2\,dx} \right) \nonumber \\
&\geq (\alpha+1)\lambda_2(\mathbb{B}; 0) - \alpha \lambda_2(\mathbb{B}; -1). \label{ineq:lbm2}
\end{align}
By Proposition 5 of \cite{FreitasLaugesen1}, we have $\lambda_2(\mathbb{B}; -1)=0$, and by definition, $\lambda_2(\mathbb{B}; 0)=\mu_2(\mathbb{B})$, the second Neumann membrane eigenvalue of the unit ball. But in all dimensions $d\geq 2$, we have $\mu_2(\mathbb{B})=p_{1,1}^2>d$ by Lemma \ref{lem:LS}. Thus \eqref{ineq:lbm2} gives the desired inequality
\[
\lambda_2(\mathbb{B}; \alpha)\geq \mu_2(\mathbb{B})(\alpha+1)\geq d(1+\alpha).
\]
\end{proof}

We now use this lemma to establish bounds on $\Lambda_2(\mathbb{B}; \tau, \alpha)$.

\begin{lem}\label{lem:ULboundsball} Let $\alpha \in [-\tau,0]$. Then the second Robin eigenvalue $\Lambda_2(\mathbb{B}; \tau, \alpha)$ satisfies the bounds
\[
d(\tau+\alpha)\leq \Lambda_2(\mathbb{B}; \tau,\alpha)\leq (d+2)(\tau+\alpha).
\]
\end{lem}

\begin{proof} First say $\alpha \in (-\tau,0)$. Since eigenfunctions for $\Lambda_1(\mathbb{B}; \tau, \alpha)$ are radial by Theorem \ref{thm:groundstate}, the coordinate functions $x_k$ are valid trial functions for $\Lambda_2(\mathbb{B}; \tau, \alpha)$ for $k=1, \ldots, d$. Thus the variational characterization gives us
\[
\Lambda_2(\mathbb{B}; \tau,\alpha)\leq \frac{\int_{\mathbb{B}} \left(|D^2x_k|^2+\tau|Dx_k|^2\right)\,dx+\alpha\int_{\partial \mathbb{B}} |x_k|^2\,dS}{\int_{\mathbb{B}} |x_k|^2\,dx}.
\]
Note that $|Dx_k|=1$ and $D^2x_k=0$. Thus, we can rewrite the above as
\[
\Lambda_2(\mathbb{B}; \tau,\alpha)\int_{\mathbb{B}} |x_k|^2\,dx\leq \tau|\mathbb{B}|+\alpha\int_{\partial\mathbb{B}} |x_k|^2\,dS.
\]
Summing over $k=1,\dots,d$ gives
\[
\Lambda_2(\mathbb{B}; \tau,\alpha)\int_{\mathbb{B}}|x|^2\,dx\leq \tau d|\mathbb{B}|+\alpha\int_{\partial \mathbb{B}} |x|^2\,dS.
\]
We also have
\[
\int_{\mathbb{B}} |x|^2\,dx=\int_{\mathbb{B}} r^2\,dx=\int_0^1r^{d+1}\,dr\int_{\partial \mathbb{B}}1\,dS=\frac{|\partial\BB|}{d+2}
\]
and
\[
|\BB|=\int_\BB 1\,dx=\int_0^1r^{d-1}\,dr\int_{\partial \mathbb{B}}1\,dS=\frac{|\partial\BB|}{d}.
\]
Thus we have
\[
\Lambda_2(\mathbb{B}; \tau, \alpha)\leq \tau \frac{d |\BB|}{(d+2)^{-1}|\partial\BB|}+\alpha\frac{|\partial\BB|}{(d+2)^{-1}|\partial\BB|}=(d+2)(\tau+\alpha).
\]

For the lower bound, we use the variational characterization for $\lambda_2(\mathbb{B}; \alpha/\tau)$, noting
\begin{align*}
\tau \lambda_2(\mathbb{B}; \alpha/\tau)&=\tau\inf_{\substack{u\in H^1(\mathbb{B})\\ u=RY_1}}\frac{\int_{\mathbb{B}}|Du|^2\,dx+\frac{\alpha}{\tau}\int_{\partial \mathbb{B}}u^2\,dS}{\int_{\mathbb{B}} u^2\,dx}\\
&\leq \inf_{\substack{u\in H^1(\mathbb{B})\\ u=RY_1}}\frac{\int_{\mathbb{B}}\left(|D^2u|^2+\tau|Du|^2\right)\,dx+\alpha \int_{\partial \mathbb{B}}u^2\,dS}{\int_{\mathbb{B}} u^2\,dx}\\
&\leq \inf_{\substack{u\in H^2(\mathbb{B})\\ u=RY_1}}\frac{\int_{\mathbb{B}}\left(|D^2u|^2+\tau|Du|^2\right)\,dx+\alpha \int_{\partial \mathbb{B}}u^2\,dS}{\int_{\mathbb{B}} u^2\,dx}\\
&=\Lambda_2(\mathbb{B}; \tau,\alpha).
\end{align*}
We note that $\alpha/\tau\in[-1,0]$ and apply the bound on $\lambda_2(\mathbb{B}; \alpha/\tau)$ from Lemma~\ref{lem:RobinMembraneBound} to complete the proof for $\alpha \in (-\tau,0)$. The general result follows from the continuity of $\Lambda_2(\mathbb{B}; \tau, \alpha)$ in $\alpha$ courtesy of Proposition \ref{Prop:Cont}. 
\end{proof}

\begin{remark}\label{rmk:ezero}
Taking $\alpha=-\tau$ in the above lemma gives that $\Lambda_2(\mathbb{B}; \tau, -\tau)=0$. Paired with Theorem \ref{thm:groundstate}, we see that for $\alpha \in [-\tau,0)$, $\Lambda_2(\mathbb{B}; \tau, \alpha)=0$ if and only if $\alpha=-\tau$.
\end{remark}

Our next lemma gives bounds on both $\tau$ and $\tau+\alpha$ in terms of the dimension $d$ and $a$. This result is essentially an immediate consequence of the previous lemma.
\begin{lem}\label{lem:boundsatb} Suppose $\alpha\in (-\tau,0)$, and $a,b$ are given as in Theorem \ref{thm:groundstate} by $a^2b^2=\Lambda_2(\mathbb{B}; \tau, \alpha)$, $b^2=a^2+\tau$. Then we have the following bounds, with the lower bounds valid for all cases under consideration, and the upper bounds valid only when $a^2<d$:
\[
\frac{a^2(a^2-\alpha)}{d+2-a^2} \leq \tau+\alpha \leq \frac{a^2(a^2-\alpha)}{d-a^2},
\]
\[
\frac{a^4-(d+2)\alpha}{d+2-a^2} \leq \tau \leq \frac{a^4-d\alpha}{d-a^2},
\]
\[
\frac{(d+2)(a^2-\alpha)}{d+2-a^2} \leq b^2\leq  \frac{d(a^2-\alpha)}{d-a^2}.
\]
\end{lem}
\begin{proof} To establish the first string of inequalities, we start with $d(\tau+\alpha)\leq \Lambda_2(\mathbb{B}; \tau, \alpha)\leq (d+2)(\tau+\alpha)$ from Lemma \ref{lem:ULboundsball}. Substituting $\Lambda_2(\mathbb{B}; \tau, \alpha)=a^2(\tau+a^2)=a^2(\tau+\alpha)+a^2(a^2-\alpha)$, we solve this string of inequalities for $\tau+\alpha$,  and use the inequality $a^2<p_{1,1}^2<d+2$ from Lemma \ref{lem:LS}. To obtain the second string of inequalities, we solve the first string of inequalities for $\tau$. Finally, we add $a^2$ to all three parts of the inequality for $\tau$ to obtain the bounds on $b^2$.
\end{proof}

We use this lemma to prove:

\begin{lem}[Large $\tau+\alpha$ Case]\label{lem:largeta} Suppose $\alpha \in (-\tau, 0)$. With $a$ as in Theorem \ref{thm:groundstate}, we have
\begin{equation}\label{eqn:largeta}
\tau+\alpha-\frac{3a^2}{d+2}> 0
\end{equation}
for all $\tau,\alpha$ such that either $a^2>(3+\alpha)(d+2)/(d+5)$ or $\tau+\alpha>3(3+\alpha)/(d+5)$.
\end{lem}

\begin{proof}
First, apply the lower bound on $\tau+\alpha$ from Lemma~\ref{lem:boundsatb}, which gives us
\begin{align}
\tau+\alpha-\frac{3a^2}{d+2}&\geq \frac{a^4-a^2\alpha}{d+2-a^2}-\frac{3a^2}{d+2} \nonumber \\
&=\frac{(d+5)a^2-(d+2)(3+\alpha)}{(d+2)(d+2-a^2)}a^2. \label{eqn:LargetaLB}
\end{align}
Thus, to prove our desired inequality, it suffices to prove the rational expression in \eqref{eqn:LargetaLB} is positive. Since $a^2<d+2$ holds by Lemma \ref{lem:LS}, this expression is positive whenever the numerator is. This occurs whenever
\[
a^2>\frac{(3+\alpha)(d+2)}{d+5}.
\]
Thus \eqref{eqn:largeta} holds whenever the above inequality holds for $a^2$. Suppose now that $a^2\leq (3+\alpha)(d+2)/(d+5)$; then
\begin{align*}
\tau+\alpha-\frac{3a^2}{d+2}&\geq \tau+\alpha-\frac{3(3+\alpha)(d+2)}{(d+2)(d+5)}\\
&=\tau+\alpha-\frac{3(3+\alpha)}{d+5}.
\end{align*}
This is last expression is positive exactly when $\tau+\alpha>3(3+\alpha)/(d+5)$.
\end{proof}

The next lemma addresses the small $\tau+\alpha$ case and complements Lemma \ref{lem:largeta}.
\begin{lem}[Small $\tau+\alpha$, Larger Dimensions]\label{lem:smallta} 
Again say $\alpha \in (-\tau,0)$ and let $a, b,$ and $\gamma$ be as in Theorem \ref{thm:groundstate}. Suppose that both $a^2\leq(3+\alpha)(d+2)/(d+5)$ and $\tau+\alpha\leq3(3+\alpha)/(d+5)$. Then for all $r\in(0,1]$,
\begin{equation}\label{ineq:nice}
\left(\tau+\alpha-\frac{3a^2}{d+2}\right)j_1(ar)+\gamma\left(\tau+\alpha+\frac{3b^2}{d+2}\right)i_1(br)>0.
\end{equation}
\end{lem}
\begin{proof}
Observe that our bound on $\tau+\alpha$ implies that $3+\alpha>0$. Next note that our bound on $a^2$ yields
\[
a^2\leq\frac{(3+\alpha)(d+2)}{d+5}< 3+\alpha<3.
\]
When $d=2$, we have $a^2\leq(3+\alpha)4/7\leq12/7<2$. So $a^2<d$ for all dimensions $d\geq 2$. This means that the upper bounds on $\tau+\alpha$ and $b^2$ from Lemma \ref{lem:boundsatb} hold. Also, notice that
\[
b^2=\tau+\alpha+a^2-\alpha \leq \frac{3(3+\alpha)}{d+5}+\frac{(3+\alpha)(d+2)}{d+5}-\alpha=3.
\]
These bounds on $a^2$ and $b^2$ mean that we can apply the bounds on $j_1''$ and $i_1''$ from Lemma~\ref{lem:besselprops}. 
We also establish a lower bound on $\gamma$ that we will use later on in this proof:
\begin{align*}
\gamma&=-\frac{a^2j_1''(a)}{b^2i_1''(b)} &&\text{by choice of $\gamma$}\\
&\geq\frac{a^3d_1-a^5d_2}{b^3d_1+\frac{6}{5}b^5d_2} &&\text{by Lemma~\ref{lem:besselprops}, with $d_1,d_2$ defined there}\\
&=\frac{a^3}{b^3} \frac{1-ca^2}{1+\frac{6}{5}cb^2} &&\text{where $c=d_2/d_1=5/(6(d+4))$}\\
&=\frac{a^3}{b^3} \frac{6(d+4)-5a^2}{6(d+4)+6b^2} &&\text{simplifying via the value of $c$}\\
&\geq\frac{a^3}{b^3} \frac{(6(d+4)-5a^2)(d-a^2)}{6(d+4)(d-a^2)+6d(a^2-\alpha)} &&\text{from the UB on $b^2$ from Lemma~\ref{lem:boundsatb}}\\
&=\frac{a^3}{b^3} \frac{(6(d+4)-5a^2)(d-a^2)}{6d(d+4-\alpha)-24a^2}=:\gamma_{\text{LB}}.
\end{align*}
We now derive a sufficient condition for inequality~\eqref{ineq:nice}.  Divide through by $i_1(br)$; we now aim to show
\[
\left(\tau+\alpha-\frac{3a^2}{d+2}\right)\frac{j_1(ar)}{i_1(br)}+\gamma\left(\tau+\alpha+\frac{3b^2}{d+2}\right)>0.
\]
We have from Lemma~\ref{lem:besselprops} that $j_1(ar)/i_1(br)\leq a/b$ for all $r\in[0^+,1]$. If the coefficient $\left(\tau+\alpha-\frac{3a^2}{d+2}\right)$ is nonnegative, there is nothing to prove, so we may assume it is negative. It thus suffices to prove
 \[
 \left(\tau+\alpha-\frac{3a^2}{d+2}\right)\frac{a}{b}+\gamma\left(\tau+\alpha+\frac{3b^2}{d+2}\right)>0.
 \]
Solving this for $\tau+\alpha$, we have
\[
\tau+\alpha>\frac{3(a^3-b^3\gamma)}{(d+2)(a+b\gamma)}.
\]
Using the lower bound on $\tau+\alpha$ from Lemma~\ref{lem:boundsatb}, we see that a sufficient condition for this to be satisfied is if
\[
\frac{a^2(a^2-\alpha)}{d+2-a^2}>\frac{3(a^3-b^3\gamma)}{(d+2)(a+b\gamma)}.
\]
This in turn can be solved for $\gamma$, yielding
\[
\gamma>\frac{a^3}{b^3}\left(\frac{3(d+2-a^2)-(d+2)(a^2-\alpha)}{\frac{a^2}{b^2}(a^2-\alpha)(d+2)+3(d+2-a^2)}\right).
\]
We apply the upper bound on $b^2$ from Lemma~\ref{lem:boundsatb} to the $a^2/b^2$ term in the denominator, which gives us
\[
\text{RHS }\leq\frac{a^3}{b^3}\left(\frac{3(d+2-a^2)-(d+2)(a^2-\alpha)}{a^2(d-a^2)\left(\frac{d+2}{d}\right)+3(d+2-a^2)}\right)=:\gamma^*.
\]
Since $\gamma \geq \gamma_{\text{LB}}$, it suffices to show that $\gamma_{\text{LB}}>\gamma^*$. Thus we consider:

\begin{align*}
\gamma_{\text{LB}}-\gamma^*&=\frac{a^3}{b^3}\left( \frac{(6(d+4)-5a^2)(d-a^2)}{6d(d+4-\alpha)-24a^2}-\frac{3(d+2-a^2)-(d+2)(a^2-\alpha)}{a^2(d-a^2)\left(\frac{d+2}{d}\right)+3(d+2-a^2)}\right)\\
&=\frac{a^3}{b^3}\left(\frac{A}{(6d(d+4-\alpha)-24a^2)(a^2(d-a^2)\left(\frac{d+2}{d}\right)+3(d+2-a^2))}\right),
\end{align*}
where the numerator $A$ is given by
\begin{align*}
\left(6(d+4)-5a^2\right)(d-a^2)\left(a^2(d-a^2)\left(\frac{d+2}{d}\right)+3(d+2-a^2)\right)-\left(3(d+2-a^2)-(d+2)(a^2-\alpha)\right)\left(6d(d+4-\alpha)-24a^2\right).
\end{align*}

Note that $a^3/b^3>0$. Additionally, since $d>a^2$, the second denominator term $a^2(d-a^2)(d+2)/d+3(d+2-a^2)$ is positive. The other denominator term satisfies
\[
6d(d+4-\alpha)-24a^2>6d(d+4-\alpha)-24d=6d(d-\alpha),
\]
and hence is also positive. Thus, it suffices to show that $A$ is positive. To do this, we first set $x=a^2$ and multiply by $d$ to clear the denominator; then
\begin{align*}
Ad&=(6(d+4)-5x)(d-x)(x(d-x)(d+2)+3d(d+2-x))-d(3(d+2-x)-(d+2)(x-\alpha))(6d(d+4-\alpha)-24x)\\
&=-5(d+2)x^4+(16d^2+41d+48)x^3-d(17d^2+58d+114)x^2\\
&\qquad+3d(4d^3+(13-2\alpha)d^2+2(5-\alpha)d+16\alpha)x-6\alpha d^2(d+2)(d+1-\alpha)\\
&=:p_\alpha(x).
\end{align*}
First, we aim to reduce dependence of $p$ on $\alpha$. Note that
\[
\frac{\partial^2}{\partial \alpha^2}p_\alpha(x)=12d^2(d+2)>0,
\]
so the first derivative $q_\alpha(x):=\partial p_\alpha/\partial \alpha$ is increasing in $\alpha$. Thus, for $x\in[0,d]$, we have
\begin{align*}
q_\alpha(x)&\leq q_0(x) =-6d^2(d+1)x +48dx-6 d^2(d+1)(d+2)\\
&\leq 0+ 48d^2-6d^2(d^2+3d+2) \\
&=-6 d^2 (d^2+3d-6)
\end{align*}
The quadratic $d^2+3d-6$ has roots at $d\approx-4.37,1.37$, so $-6d(d^2+3d-6)<0$ for all $d\geq 2$. So $q_\alpha(x)<0$. This means that $p_\alpha(x)$ is decreasing in $\alpha$, and hence is minimized at $\alpha=0$. Note the constant term in $p_\alpha(x)$ vanishes when $\alpha=0$, so we may now consider
\[
P(x):=p_0(x)/x=-5(d+2)x^3+(16d^2+41d+48)x^2-d(17d^2+58d+114)x+3d^2(d+2)(4d+5).
\]
We aim to show $P(x)>0$ on $[0,d]$. Again, we look at higher derivatives:
\[
P''(x)= -30(d+2)x+2(16d^2+41d+48).
\]
This function is clearly decreasing in $x$, so
\[
P''(x)\geq P''(d)=2 d^2+22d+96,
\]
and hence $P'(x)$ is increasing in $x$. But then $P'(x)\leq P'(d)=- 6 d^2-18d <0$, so $P$ is decreasing in $x$. Thus $P$ is minimized when $x=d$, so we look at $P(d)=6d^2(d^2+2d-6)$, which has roots $d=0,-1\pm\sqrt{7}$. Thus $P(d)>0$ for all $d\geq 2>-1+\sqrt{7}$.  This means that $Ad=p_\alpha(x)>0$ since $0<x=a^2<d$, and hence $\gamma_{\text{LB}}-\gamma^*>0$. But that is a sufficient condition for inequality~\eqref{ineq:nice} to be satisfied, so we are done.
\end{proof}

\section{Proof of the Isoperimetric Inequality}

In this, our final section, we prove the paper's main result and related corollaries. We begin with a preliminary tool that allows us to reduce the proof of Theorem \ref{thm:mainthm} to the case where $|\Omega|=|\mathbb{B}|$.

\begin{lem} \label{lem:scaling}
The eigenvalues of problem \eqref{eqn:robinplate} scale according to
\[
\Lambda(\Omega;\tau,\alpha)=t^{-4}\Lambda(t^{-1}\Omega;t^2\tau,t^3\alpha).
\]
\end{lem}

\begin{proof}
Let $u$ be a valid trial function on $\Omega$, and $t>0$ a positive scaling constant. Then $\tu:=u(x/t)$ is a valid trial function on $t\Omega$.
Looking at the numerator of the Rayleigh quotient,
\begin{align*}
N[t\Omega; \tau, \alpha, \tu]&=\int_{t\Omega}\left(|D^2\tu|^2+\tau|D\tu|^2\right)\,dx+\alpha\int_{\partial(t\Omega)} \tu^2\,dS_x\\
&=\int_{t\Omega}\left(|t^{-2}(D^2u)(x/t)|^2+\tau|t^{-1}(Du)(x/t)|^2\right)\,dx+\alpha\int_{\partial(t\Omega)}u(x/t)^2\,dS_x\\
&=t^d\int_{\Omega}\left(t^{-4}|D^2u|^2+\tau t^{-2}|Du|^2\right)\,dy+t^{d-1}\alpha\int_{\partial\Omega}u^2\,dS_y,
\end{align*}
with the last line following from the change of variable $y=x/t$, $dy=t^{-d}dx$.
Similarly (and much more simply), we can rewrite the denominator as
\[
\int_{t\Omega}\tu^2\,dx=t^d\int_\Omega u^2\,dy.
\]
Thus we have the following relationship between the Rayleigh quotients:
\begin{align*}
Q[t\Omega; \tau, \alpha, \tu]&=\frac{N[t\Omega; \tau,\alpha,\tu]}{\int_{t\Omega}\tu^2\,dx}\\
&=\frac{t^d\int_{\Omega}\left(t^{-4}|D^2u|^2+\tau t^{-2}|Du|^2\right)\,dy+t^{d-1}\alpha\int_{\partial\Omega}u^2\,dS_y}{t^d\int_\Omega u^2\,dy}\\
&=\frac{t^{-4}\int_{\Omega}\left(|D^2u|^2+\tau t^{2}|Du|^2\right)\,dy+t^{-4}(\alpha t^3) \int_{\partial\Omega}u^2\,dS_y}{\int_\Omega u^2\,dy}\\
&=\frac{t^{-4}N[\Omega; t^2\tau, t^3\alpha, u]}{\int_\Omega u ^2\,dy}\\
&=t^{-4}Q[\Omega; t^2\tau, t^3\alpha, u].
\end{align*}
\end{proof}

We continue to consider $\Omega$ as in Theorem \ref{thm:mainthm} normalized so that $|\Omega|=|\mathbb{B}|$ and keep $\alpha \in (-\tau,0)$. We follow Weinberger's approach \cite{Weinberger} to prove our isoperimetric inequality. Recall that our trial functions for $\Lambda_2(\Omega; \tau, \alpha)$ take the form $u_k=\rho x_k/r$, with $\rho$ defined in \eqref{eq:pdef}. By Lemma \ref{lem:COM}, we assume that our domain $\Omega$ has been suitably translated to ensure our trial functions are orthogonal to a ground state eigenfunction $v$ with nonzero mean corresponding to $\Lambda_1(\Omega;\tau,\alpha)$ (in our proof of Theorem \ref{thm:mainthm}, we separately dispense of the case where $v$ has zero mean). Then by the variational characterization of $\Lambda_2(\Omega;\tau,\alpha)$, we have
\[
\Lambda_2(\Omega;\tau,\alpha)=\inf_{\substack{u\in H^2(\Omega)\\ u\perp v}}Q[\Omega; \alpha, u]\leq Q[\Omega; \alpha, u_k].
\]
Recalling the definition of the Rayleigh quotient $Q$, we can rewrite the inequality above as
\[
\Lambda_2(\Omega;\tau,\alpha)\int_\Omega |u_k|^2\,dx \leq \int_\Omega \left(|D^2u_k|^2+\tau|Du_k|^2\right)\,dx+\alpha \int_{\partial\Omega}u_k^2\,dS.
\]
Summing over all $k=1,\dots,d$, the left-hand side can be easily simplified. The sum for the right-hand side is more computationally onerous. Fortunately, the volume integral terms are treated in \cite{chasmanineq} (see p.437) and the surface integral term is treated in Proposition 1 of \cite{FreitasLaugesen1}. Note that \cite{chasmanineq} considers trial functions of the form (radial function) times $x_k/r$, so the algebra remains valid. We then have
\[
\Lambda_2(\Omega;\tau,\alpha)\int_\Omega \rho^2\,dx \leq \int_\Omega N[\rho]\,dx,
\]
where $N[\rho]$ is defined as
\[
N[\rho]=(\rho'')^2+\frac{3(d-1)}{r^4}(\rho-r\rho')^2+\tau\left((\rho')^2+\frac{d-1}{r^2}\rho^2\right)+\alpha\left(2\rho\rho'+\frac{d-1}{r}\rho^2\right).
\]

Note that by Lemma~\ref{lem:trialprops},
\[
\frac{d}{dr}\rho(r)^2=2\rho(r)\rho'(r)\geq0\quad\text{on $[0,\infty)$},
\]
so $\rho^2$ is increasing in $r$, and thus we observe the following monotonicity result:
\[
\int_\Omega \rho^2\,dx\geq \int_{\mathbb{B}} \rho^2\,dx.
\]

In our proof of Theorem \ref{thm:mainthm}, we show that $N[\rho]$ satisfies a \emph{partial monotonicity} condition. We call a function $F:\mathbb{R}^d \to \mathbb{R}$ partially monotonic on $\Omega$ if
\begin{equation}\label{cond:PM}
F(x)>F(y) \quad \textup{whenever }x\in \Omega \textup{ and } y\notin \Omega.
\end{equation}

Our goal is now to show:
\begin{lem}\label{lem:monotonicitygoal}
The function $N[\rho]$ is partially monotonic on $\mathbb{B}$.
\end{lem}
\begin{proof}
Drawing from both the free plate and the Robin membrane, we group the terms of $N[\rho]$ as follows:
\begin{align*}
N_1(r)&=(\rho'')^2\\
N_2(r)&=\frac{3}{r^4}(\rho-r\rho')^2+\tau \frac{\rho^2}{r^2}+\alpha\frac{\rho^2}{r}\\
N_3(r)&=\tau (\rho')^2+2\alpha\rho\rho',
\end{align*}
so that we have $N[\rho]=N_1(r)+(d-1)N_2(r)+N_3(r)$.

Let us address $N_1$ first. Recall that $\rho$ is linear for $r\geq1$; therefore $\rho''$ vanishes in this region. By Lemma \ref{lem:trialprops}, $\rho'' \leq 0$ on $[0,1]$. Moreover, $\rho'' < 0$ on $(0,1)$, since $\rho''$ is strictly convex on $(0,1)$ with $\rho''(0)=\rho''(1)=0$. We conclude that $(\rho'')^2$ satisfies \eqref{cond:PM} whenever $x\neq 0$ (when $x=0$ we have equality). In the remainder of the proof, we show that $(d-1)N_2(r)+N_3(r)$ is decreasing on $[1,\infty)$, and that $N_2(r)$ and $N_3(r)$ are decreasing on $[0,1]$ with $N_3(r)$ strictly decreasing near $r=0$. 

Consider the terms $(d-1)N_2(r)+N_3(r)$ together for $[1,\infty)$. Recall $\rho$ is linear on this interval. We  write $\rho(r)=Ar+B$, where $A=\rho'(1)$ and $B=\rho(1)-\rho'(1)$ are both nonnegative by Lemma~\ref{lem:trialprops}. Then
\begin{align*}
(d-1)N_2(r)+N_3(r)&=\frac{3(d-1)B^2}{r^4}+\frac{\tau(d-1)B^2}{r^2}+\frac{d-1}{r}(\alpha B^2+2AB\tau)+d(\tau A^2+2\alpha A B)+(d+1)A^2\alpha r\\
(d-1)N_2'(r)+N_3'(r)&=-\frac{12(d-1)B^2}{r^5}-\frac{2\tau(d-1)B^2}{r^3}-\frac{d-1}{r^2}(\alpha B^2+2AB\tau)+(d+1)A^2\alpha.
\end{align*}
Since $\rho'(1)>0$ and $\alpha<0$, $(d-1)N_2(r)+N_3(r)$  eventually becomes negative and decreasing as $r\to\infty$ . We also note that $(d-1)N_2'(r)+N_3'(r)$ is negative whenever $\alpha B+2A\tau\geq0$. But by Lemma~\ref{lem:trialprops}, we have $\tau\rho'(1)+\alpha\rho(1)=(\tau+\alpha)A+\alpha B\geq0$, so
\[
\alpha B+2A\tau\geq -(\tau+\alpha)A+2A\tau=(\tau-\alpha)A\geq 0.
\]
Thus $(d-1)N_2(r)+N_3(r)$ is decreasing on $[1,\infty)$.

We next consider $N_2(r)$ and $N_3(r)$ separately for $[0,1]$. First, note that
\begin{align*}
N_3'(r)&=2\tau\rho'\rho''+2\alpha\rho\rho''+2\alpha(\rho')^2\\
&=2\rho''(\alpha\rho+\tau\rho')+2\alpha(\rho')^2,
\end{align*}
which is nonpositive since $\alpha<0$ and $\alpha\rho+\tau\rho'\geq0$ by Lemma~\ref{lem:trialprops}. Thus $N_3$ is decreasing on $[0,1]$. Moreover, since $\rho'(0)>0$, $N_3(r)$ is strictly decreasing near $r=0$.

To show $N_2$ is decreasing as a function of $r$ for $r\in[0,1]$, we again differentiate:
\begin{align*}
N_2'(r)&=-\frac{2}{r^3}(\rho-r\rho')\left(\frac{6}{r^2}(\rho-r\rho')+3\rho''+\tau\rho\right)-\alpha\frac{\rho^2}{r^2}+\frac{2\alpha\rho\rho'}{r}\\
&=-\frac{2}{r^3}(\rho-r\rho')\left(\frac{6}{r^2}(\rho-r\rho')+3\rho''+\tau\rho+\frac{\alpha r\rho}{2}\right)+\alpha\frac{\rho \rho'}{r}
\end{align*}

Since $\rho-r\rho'$ is nonnegative by Lemma \ref{lem:trialprops}, and the coefficient $\alpha<0$, the above expression will be nonpositive once we establish
\[
n_2(r):=\frac{6}{r^2}(\rho-r\rho')+3\rho''+\tau\rho+\frac{\alpha r\rho}{2}\geq0
\]
To do this, note that $\alpha<0$, $\rho\geq0$, and $r\in[0,1]$ give us
\[
n_2(r)\geq \frac{6}{r^2}(\rho-r\rho')+3\rho''+(\tau+\alpha)\rho.
\]
Recall that on $[0,1]$, our radial function $\rho(r)=j_1(ar)+\gamma i_1(br)$. Then we have (see p.440 of \cite{chasmanineq}, for instance)
\begin{align*}
\frac{6}{r^2}(\rho-r\rho')+&3\rho''+(\tau+\alpha)\rho\\
&=\left(\tau+\alpha-\frac{3a^2}{d+2}\right)j_1(ar)+\gamma\left(\tau+\alpha+\frac{3b^2}{d+2}\right)i_1(br)\\
&\qquad+\frac{3(d+1)}{d+2}\left(a^2j_3(ar)+\gamma b^2 i_3(br)\right).
\end{align*}
The modified Bessel functions $i_1$ and $i_3$ are nonnegative everywhere and, since $\tau+\alpha\geq0$, their coefficients are nonnegative as well. The Bessel functions $j_1(ar)$ and $j_3(ar)$ are nonnegative for $r\in[0,1]$ by Lemma~\ref{lem:besselprops}, and the coefficient of $j_3$ is also clearly nonnegative. Thus we must address the coefficient of $j_1$.

For $a^2>(3+\alpha)(d+2)/(d+5)$ or $\tau+\alpha>3(3+\alpha)/(d+5)$, this coefficient is nonnegative by Lemma~\ref{lem:largeta}. When $a^2\leq (3+\alpha)(d+2)/(d+5)$ and $\tau+\alpha\leq 3(3+\alpha)/(d+5)$, the first line is nonnegative on $[0,1]$ by Lemma~\ref{lem:smallta}. Thus the sum is nonnegative, and so $n_2\geq0$. We have therefore shown $N_2(r)$ is decreasing on $[0,1]$.

Taken in sum, our work shows that $N[\rho]$ is partially monotonic on $\mathbb{B}$.
\end{proof}

We are now prepared to prove our main result.

\begin{proof}[Proof of Theorem \ref{thm:mainthm}]
First suppose $|\Omega|=|\mathbb{B}|$ and $\alpha \in (-\tau,0)$ with $v$ our fixed eigenfunction for $\Lambda_1(\Omega; \tau, \alpha)$. If it so happens that $v$ has zero mean, then the constant function $1$ is a valid trial function in the Rayleigh quotient for $\Lambda_2(\Omega; \tau, \alpha)$ and so
\[
\Lambda_2(\Omega; \tau, \alpha)\leq \alpha \frac{\int_{\partial \Omega}1\,dS}{\int_{\Omega}1\,dx}=\alpha \frac{|\partial \Omega|}{|\Omega|}<0.
\]
By Theorem \ref{thm:groundstate}, we have $\Lambda_2(\mathbb{B}; \tau, \alpha)>0$, and so
\[
\Lambda_2(\Omega; \tau, \alpha) < \Lambda_2(\mathbb{B}; \tau, \alpha)
\]
in this case. Thus, we may assume that the eigenfunction $v$ has nonzero mean. From Lemma \ref{lem:monotonicitygoal} and the discussion that precedes it, we have
\begin{align*}
\Lambda_2(\Omega; \tau, \alpha) & \leq\frac{\int_\Omega N[\rho]\,dx}{\int_\Omega \rho^2\,dx}\\
&\leq\frac{\int_{\Omega} N[\rho]\,dx}{\int_{\mathbb{B}} \rho^2\,dx}\\
&= \frac{\int_{\Omega \cap \mathbb{B}} N[\rho]\,dx+\int_{\Omega \setminus \mathbb{B}}N[\rho]\,dx}{\int_{\mathbb{B}} \rho^2\,dx}\\
&\leq \frac{\int_{\Omega \cap \mathbb{B}} N[\rho]\,dx+|\Omega \setminus \mathbb{B}|\sup_{\Omega \setminus \mathbb{B}}N[\rho]}{\int_{\mathbb{B}} \rho^2\,dx}\\
&\leq \frac{\int_{\Omega \cap \mathbb{B}} N[\rho]\,dx+|\mathbb{B} \setminus \Omega|\inf_{\mathbb{B} \setminus \Omega}N[\rho]}{\int_{\mathbb{B}} \rho^2\,dx}\\
&\leq \frac{\int_{\Omega \cap \mathbb{B}} N[\rho]\,dx+\int_{\mathbb{B}\setminus \Omega}N[\rho]\,dx}{\int_{\mathbb{B}} \rho^2\,dx}\\
&=\frac{\int_{\mathbb{B}} N[\rho]\,dx}{\int_{\mathbb{B}} \rho^2\,dx}\\
&=\Lambda_2(\mathbb{B}; \tau, \alpha),
\end{align*}
where the fourth inequality relies on partial monotonicity. Moreover, since $N[\rho]$ is partially monotonic and $\Omega$ is a $C^{\infty}$ domain, equality holds if and only if $\Omega=\mathbb{B}$. The isoperimetric inequality holds at $\alpha=-\tau,0$ by the continuity of the eigenvalues in $\alpha$.

When $\Omega$ is a general domain, choose $t$ so that $|t^{-1}\Omega|=|\mathbb{B}|$. Then since $-\tau/t<\alpha<0$, we deduce
\[
t^{-4}\Lambda_2(t^{-1}\Omega; t^2\tau, t^3\alpha)\leq t^{-4}\Lambda_2(\mathbb{B}; t^2\tau, t^3\alpha),
\]
which implies $\Lambda_2(\Omega; \tau, \alpha)\leq \Lambda_2(t\mathbb{B}; \tau, \alpha)$ by Lemma \ref{lem:scaling}. The sharpness of the inequality is inherited from the normalized version, and the inequality holds at $\alpha=-\tau/t,0$ again by continuity.
\end{proof}

\begin{proof}[Proof of Corollary \ref{Cor:MainCor}]
Taking $\alpha=0$ in Theorem \ref{thm:mainthm}, one immediately obtains
\[
\Lambda_2(\Omega; \tau, 0)=\omega_2(\Omega)\leq \omega_2(\Omega^*)=\Lambda_2(\Omega; \tau, 0).
\]
To prove the second statement, we assume $\Omega$ is normalized so that $|\Omega|=|\mathbb{B}|$. Observe that $\sigma$ is an eigenvalue of problem \eqref{eqn:BiharmonicSteklov} precisely when $0$ is an eigenvalue of the Robin problem \eqref{eqn:robinplate} with $\alpha=-\sigma$. By taking $\alpha=-\tau$ in Theorem \ref{thm:mainthm}, we obtain
\[
\Lambda_2(\Omega; \tau, -\tau)\leq \Lambda_2(\mathbb{B}; \tau, -\tau)=0,
\]
where the last equality holds by Remark \ref{rmk:ezero}. On the other hand, $\Lambda_2(\Omega; \tau, 0)=\omega_2(\Omega)>0$. Thus we are led to define
\[
\alpha'=\sup \{\alpha:\Lambda_2(\Omega; \tau, \alpha)=0\}.
\]
Since $\Lambda_2(\Omega; \tau, \alpha)$ is continuous with respect $\alpha$, we have $-\tau \leq \alpha'<0$. Then $\sigma=-\alpha'$ is an eigenvalue of problem \eqref{eqn:BiharmonicSteklov}, and by construction $\sigma_2(\Omega)=-\alpha'$. Thus
\begin{equation}\label{Ineq:IsoStek}
\sigma_2(\Omega)=-\alpha'\leq \tau=\sigma_2(\mathbb{B}),
\end{equation}
where the last equality holds by Remark \ref{rmk:ezero}.
\end{proof}

\section*{Acknowledgements}
Chasman thanks the Bucknell University Mathematics Department for providing housing during a visit during the spring of 2018, when conversations about this project first began. Both authors thank Richard Laugesen for encouraging us to pursue this problem and for useful discussions.

\end{document}